\documentclass[10pt,reqno,twoside]{amsart}
\usepackage{amsmath,amsfonts,amsthm,amssymb,amscd,stmaryrd,url,amstext, amsxtra,amsopn }

\DeclareFontFamily{OT1}{rsfs}{}
\DeclareFontShape{OT1}{rsfs}{n}{it}{<-> rsfs10}{}
\DeclareMathAlphabet{\mathscr}{OT1}{rsfs}{n}{it}
 \newcommand{\ds}{\displaystyle}

\newcommand{\C}{{\mathbb{C}}}

\newcommand{\Lo}{{\mathscr{L}}}

\renewcommand{\Re}{{\mathfrak{Re}}}
\renewcommand{\Im}{{\mathfrak{Im}}}

\newtheorem{theorem}{Theorem}
 \newtheorem{corollary}{Corollary}[theorem]
 \newtheorem{lemma}[theorem]{Lemma}

  \theoremstyle{remark}

\reversemarginpar 

\begin{document}
\title{Explicit zero density theorems for Dedekind zeta functions}
\author{Habiba Kadiri and Nathan Ng}
\maketitle

{\def\thefootnote{}
}


\section{Introduction}

This article concerns the zeros of Dedekind zeta functions. 
We prove bounds 
for the number of zeros of Dedekind zeta functions in boxes and we prove zero repulsion theorems
for these zeros.  The formula for zeros in boxes is classical and dates back to Riemann and von Mangoldt.
The zero repulsion property is commonly referred to as the Deuring-Heilbronn phenomenon.
It asserts that if an $L$-function has a zero very close to $s=1$ then the other zeros of this $L$-function 
are pushed further away from $s=1$.  Theorems of this type already exist in the literature, but they do not give 
explicit constants.  Our theorems will determine such constants and this is important in our related work 
which concerns an explicit bound for the least prime ideal in Chebotarev's density theorem.  
Let $K_0$ be a number field and $K$ a Galois extension of $K_0$ with ring of integers $\mathcal{O}_K$. Its 
degree is denoted $n_K = [K : \mathbb{Q}]$ and its absolute discriminant is $d_K$. 
Let $G$ be the Galois group of $K/K_0$, and let $C \subset G$ be a conjugacy class.
We show in \cite{KN} that there exists an unramified prime ideal $\mathfrak{p}$ of degree one such that its Frobenius $\sigma_{\mathfrak{p}}=C$ and its norm $\mathbb{N} \mathfrak{p} \le d_K^{C_0}$ for an explicit constant $C_0 >0$.  This theorem makes use 
of various results concerning the location and number of zeros of the Dedekind zeta function of $K$.  

We now state our results.  
Throughout this article we shall encounter the quantity $\log d_K$.  From this point on, we shall employ the abbreviation
\begin{equation}
  \label{eq:L}
   \mathscr{L} = \log d_K. 
\end{equation}
The Dedekind zeta function of $K$ is  
\[
   \zeta_{K}(s) =
   \sum_{\mathfrak{a} \subset  \mathcal{O}_K }
\frac{1}{(\mathbb{N} \mathfrak{a})^s}
\]
where $\mathfrak{a}$ ranges through non-zero ideals. 
We now define a function which counts the zeros of $\zeta_K(s)$ in boxes.
Throughout this article we shall denote the non-trivial zeros of $\zeta_K(s)$ as $\varrho=\beta+i \gamma$
where $\beta, \gamma \in \mathbb{R}$.  We set for $T \ge 0$
\[
   N_{K}(T) = \# \{ \varrho \ | \ \zeta_K(\varrho)=0, \ 0 < \beta < 1, \ | \gamma | \le T \} .
\]

Our first result is 
\begin{theorem} \label{nt}
 Let $T \ge 1$ and $0 < \eta \le \frac{1}{2}$ then 
\[
   |N_K(T)- \frac{T}{\pi} 
   \log(\mbox{$(\frac{T}{2 \pi e})^{n_K}$}d_K    )   | 
   \le 
  c_1(\eta)  (\Lo  +  n_K \log T) + c_2(\eta) n_K  
  + 7.6227
\]
where 
    \begin{align*}
   c_1(\eta) & = \frac{1+ 2 \eta}{\pi \log 2}, \\
   c_2(\eta) 
    & = 0.2675 - 0.2680 \eta  
    + \frac{2}{\log 2} \log \frac{\zeta(1+\eta)^2}{\zeta(2 +2 \eta)}  
    +\frac{2}{\pi} \log \zeta( \tfrac{3}{2}+ 2 \eta). 
\end{align*}
\end{theorem}
These results were proven by following arguments of Backlund \cite{Ba}, Rosser \cite{Ro}, and McCurley \cite{Mc}
who obtained analogous results for the Riemann zeta function and Dirichlet $L$-functions.  Their work is 
important in arguments which give explicit zero-free regions for $L$-functions and explicit bounds for prime counting 
functions. 

Our next result is an inequality for the real part of the logarithmic derivative of $\zeta_K(s)$.   We establish 
\begin{theorem}  \label{gr}
Let $0 < \epsilon \le 10^{-2}$, $s=\sigma+it$, $\sigma >1$, $|t| \le 1$, and
\begin{equation}
\label{eq:phi}
\phi = \frac{1-\tfrac{1}{\sqrt{5}}}{2}=0.276393 \ldots
\end{equation}
We define a multiset of non-trivial zeros of $\zeta_K(s)$ by
\begin{equation}\label{eq:defRepsilont}
  \mathcal{R}_{\epsilon,t} = \{ \varrho \ |  \ \zeta_K(\varrho)=0, \ 1- \epsilon \le \beta < 1, \ |\gamma - t| \le 1 \}. 
\end{equation}
For $1 < \Re(s) \le 1+\epsilon$, we have
\begin{multline}
  \label{eq:graham}
-\Re \left( \frac{\zeta_K'}{\zeta_K}(s) \right) \le   
\Re\Big( \frac{1}{s-1}\Big) 
-  \sum_{\varrho \in \mathcal{R}_{\epsilon,t}}  \Re \Big( \frac{1}{s -\varrho} \Big)
+ \phi \Lo
 -0.0354 n_K \\ + 5 \epsilon ( |\mathcal{R}_{\epsilon,t}| +1)+ 0.1216.
\end{multline}
\end{theorem}
It is well known that  there exists an explicit constant $C_0$ such that 
$n_K \le  C_0 \Lo$ for $K  \ne \mathbb{Q}$. This follows from an inequality due to 
Minkowski:
\[
 d_K > \Big(\frac{\pi}{4} \Big)^{n_K}  \Big( \frac{n_{K}^{n_K}}{n_K!} \Big)^2 \text{ for } K  \ne \mathbb{Q}.
\]
This combined with Theorem \ref{nt} implies there exists an explicit constant $C_1$ such that 
\[
  |\mathcal{R}_{\epsilon,t}| \le N_K(2)
  \le C_1 \Lo
\]
for all $|t| \le 1$. 
\begin{corollary} 
\label{gr2}
Let $0 < \epsilon \le 10^{-2}$, $s=\sigma+it$, $\sigma >1$, and $|t| \le 1$.
If $d_K$ is sufficiently large, then there exists a positive constant $C_2$ such that
\begin{equation} 
\begin{split}
  \label{eq:graham2}
-\Re \left( \frac{\zeta_K'}{\zeta_K}(s) \right)  & \le   
\Re\Big( \frac{1}{s-1}\Big) 
-  \sum_{\varrho \in \mathcal{R}_{\epsilon,t}} 
    \Re\Big( \frac{1}{s-\varrho}\Big)  
 + (\phi +C_2 \epsilon) \Lo .
\end{split}
\end{equation}
\end{corollary}
The inequality  given in Corollary \ref{gr2} will play an important role in obtaining zero-free
regions and zero-repulsion theorems. 
Versions of this result have been proven in Graham
\cite{Gr} for Dirichlet $L$-functions and implicitly in Stechkin \cite{St} for the Riemann zeta 
function. 
Our proof, as in \cite{St}, uses the global method, namely the classical explicit formula for $-\Re (\zeta_{K}'(s)/\zeta_{K}(s))$ in conjunction
with the Stechkin differencing trick. 
In order to obtain a negative contribution in $n_K$,  we need to improve on a lemma of McCurley on $-\Re (\Gamma'(s)/\Gamma(s))$.
In \cite{HB}, Heath-Brown employs a local method using a Jensen type formula which produces much better values of $\phi$ ($\phi \le \frac{1}{6}$). Recently, Li \cite[Lemma 4]{Li}
applied this method to $\zeta_K(s)$ and was able to obtain an inequality like \eqref{eq:graham2} with $\phi = \frac{1}{4}$ and an extra term of size $2n_K  \log (\tfrac{\Lo}{n_K})+O(n_K)$.
If $n_K=o(\Lo)$, then Li's result is superior to ours.  
On the other hand, for those fields $K$ with $n_K  \gg \Lo$ this error term becomes weaker than \eqref{eq:graham2}.

In order to obtain good zero-free regions with nice constants we will develop a smooth variant of the above theorem.  
Such results have already been proven by Heath-Brown in the case of Dirichlet $L$-functions. 
We shall follow closely his approach, though there are several differences in the argument. 
Let $f$ be  a continuous function from $[0,\infty)$ to $\mathbb{R}$ and supported in $[0,x_0)$. In addition, $f$ is twice differentiable on $(0,x_0)$ 
with a bounded and continuous second derivative. 
Its associated Laplace transform is $F(z) = \int_{0}^{\infty} f(t) e^{-zt} dt$.  

\begin{theorem} \label{smooth}
Let $0 < \epsilon \le 10^{-2}$, $0 < \delta < 1$, $s=\sigma + it$ with $\sigma > 1- \frac{(1-\delta)(\log \Lo)}{ x_0 \Lo}$, and $|t| \le 1$.
Suppose $f(0) \ge 0$. 
If $d_K$ is sufficiently large, then there exists a positive constant $C_3$ such that 
\begin{multline*}
 \Re \Big(  \sum_{\mathfrak{a} \subset  \mathcal{O}_K \backslash \{ 0 \} }
    \frac{\Lambda(\mathfrak{a})}{(N\mathfrak{a})^s} f(\Lo^{-1} \log N \mathfrak{a}) \Big)
    \le \Lo \Re F((s-1)\Lo)  \\
     - \Lo\sum_{\varrho \in \mathcal{R}_{\epsilon,t}} 
  \Re( F((s-\varrho) \Lo)) 
 + f(0)(\phi+ C_3 \epsilon) \Lo.
 \end{multline*}
\end{theorem}
An advantage of the inequality in Theorem \ref{smooth} over the one in Corollary \ref{gr2} is that it allows for a wide variety 
of functions $f$ and it allows  $\sigma$ to be chosen inside the critical strip.

We now state our zero-repulsion theorems. It is known that $\zeta_K(s)$ possesses at most one real zero
in a region close to one. 
For example, it was proven in \cite{LMO} that there exists a positive constant $R$ such that 
$\zeta_K(s)$ does not vanish in 
\begin{equation}
   \label{eq:kadiri}
    \Re (s) \ge 1- \frac{1}{R \Lo} \text{ and }  |\Im(s)| \le 1
\end{equation}
with the exception of  possibly one real zero $\beta_1$. 
Recently Kadiri \cite{K} proved that $R=12.74$ is a valid constant.   We now examine the
consequences of the existence of this possible exceptional zero. 
Let 
\[
   \beta_1 = 1- \lambda_1 \Lo^{-1} \text{ where } \lambda_1 > 0.
\]
Let $\varrho'=\beta'+i \gamma'$ be another zero satisfying 
\[
  \beta'= 1- \lambda' \Lo^{-1} \text{ where } \lambda' > 0 \text{ and }
  |\gamma'| \le 1. 
\]
We shall prove 
\begin{theorem} \label{dh} 
Let $\beta_1$ be an exceptional zero of $\zeta_K(s)$ satisfying $\lambda_1 < R^{-1}$. 
Let $\varrho'=\beta'+i \gamma'$ be another zero of $\zeta_K(s)$ satisfying $|\gamma'| \le 1$, $\lambda' < \frac{1}{13.85}(\log \Lo)$, and $\beta'$ is maximal with respect to these conditions.
If  $d_K$ is sufficiently large,
then 
\[
     \lambda' \ge 0.6546  \log (\lambda_{1}^{-1}).
\]
\end{theorem}
The first proof of this type, in the case of Dirichlet $L$-functions, is due to Linnik \cite{Lin}. 
His proof was complicated; it made use of Brun's sieve and convexity theorems for entire functions.   
It should be noted that Linnik's result  plays an important role in the proof that 
the least prime in an arithmetic progression modulo $q$ is $\ll q^{C_4}$ for some positive constant $C_4$. 
Knapowski \cite{Kn} simplified the argument by applying Turan's power sum method.
Later, Motohashi \cite{Mo} and Jutila \cite{Ju} independently showed that an argument related to Selberg's sieve led to better numerical results. 
Finally, Heath-Brown \cite{HB} made significant numerical improvements by employing a smoothed version of the explicit formula. 
His corresponding theorem for Dirichlet $L$-functions has a $2$ in place of our $0.6546 $ (see Lemma 8.1 and Table 2 of \cite{HB}).
He makes use of an explicit formula for Dirichlet $L$-functions like our Theorem \ref{smooth}. It turns out that the coefficient of $\log (\lambda_{1}^{-1})$ depends on how small $\phi$ is.
This was one of many ingredients in his proof that $C_4=5.5$ is valid. 
In \cite{LMO}, Lagarias, Montgomery, and Odlyzko proved an inexplicit version of Theorem \ref{dh}.
They used a smoothing through differentiation in conjunction with a variant of Turan's power sum method.
Instead, we shall prove the above theorem by following Heath-Brown's method.

\noindent
{\it Conventions and Notation}. 
We shall use extensively big $\mathcal{O}$ notation and Linnik's notation. 
For a complex number $A$ and a real number $B$ we shall use the notation $A=\mathcal{O}(B)$ and $A \ll B$ or $B \gg A$ to  mean
there exists $M >0$ such that $|A| \le M B$ for $A$ sufficiently large.

\section{Properties of the Dedekind zeta function}

The Dedekind zeta function of $K$ possesses the Euler product
\[
   \zeta_K(s) = \prod_{\mathfrak{p}} (1-(N \mathfrak{p})^{-s})^{-1} 
\]
where $\mathfrak{p}$ ranges over all prime ideals in $\mathcal{O}_K$ and $\Re(s) >1$.
It is convenient to consider the completed zeta function
\begin{align}  \label{def-xi}
\xi_K(s) & =s(s-1) (d_K)^{s/2}    \gamma_K(s)\zeta_K(s),   \\
\gamma_K(s)  & = ( \pi^{n_K}2^{2r_2} )^{-s/2} \Gamma(s/2)^{r_1} \Gamma(s)^{r_2}, \label{eq:gamK}
\end{align}
where $r_1$ and $r_2$ are the number of real and complex places in $K$.  
The benefit of working with $\xi_K$ is that it is entire of order 1, it satisfies 
the functional equation
\begin{equation}
  \label{eq:xife}
  \xi_K(s)=\xi_K(1-s),
\end{equation}
and its zeros are the non-trivial zeros of $\zeta_K(s)$. 

\section{Proof of Theorem \ref{nt} and Theorem \ref{gr}}

\begin{proof}[Proof of Theorem \ref{nt}]
Let $0 < \eta \le \frac{1}{2}$ and define \[\sigma_1=\frac{3}{2}+2\eta.\]
Throughout this proof we shall let $\theta_j$, for $j=1, \ldots, 4$, denote real numbers which satisfy 
$|\theta_j| \le 1$. 
We follow the argument of McCurley \cite{Mc} which generalized earlier arguments of 
Backlund and Rosser. 
Assume that $\pm T$ does not coincide with the ordinate of a zero. 
We consider the rectangle $\mathcal{R}$ with vertices $\sigma_1-iT,
\sigma_1+iT, 1-\sigma_1+iT,$ and $1-\sigma_1-iT$ where $\sigma_1 > 1$.
Since $\xi_K(s)$ is entire, the argument principle yields
\begin{equation*} 
    N_K(T) = \frac{1}{2 \pi}
    \Delta_{\mathcal{R}} \arg \xi_K(s). 
\end{equation*}
Let $\mathcal{C}$ be the part of the contour $\mathcal{R}$ in $\Re(s) \ge \frac{1}{2}$
and $\mathcal{C}_0$ the part of the contour in $\Re(s) \ge \frac{1}{2}$ and $\Im(s) \ge 0$.
By the functional equation and by the formula $\overline{\xi_{K}(s)}=\xi(\overline{s})$
it follows that 
\[
     \Delta_{\mathcal{R}} \arg \xi_K(s) =2    \Delta_{\mathcal{C}} \arg \xi_K(s)=
   4
    \Delta_{\mathcal{C}_0} \arg \xi_K(s)
\]
and therefore
\begin{equation}
    \label{eq:NKT}
      N_K(T) = \frac{2}{ \pi}
    \Delta_{\mathcal{C}_0} \arg \xi_K(s).   
\end{equation}
We write 
$\xi_K(s) = s B^{\frac{s}{2}}  \Gamma( \mbox{$\frac{s}{2}$})^{r_1}
   \Gamma(s)^{r_2} (s-1) \zeta_K(s)$ where $B=  \frac{d_K}{\pi^{n_K}2^{2r_2}}$.
   Hence
\begin{multline*}
 \Delta_{\mathcal{C}_0} \arg \xi_K(s) 
   =  
       \Delta_{\mathcal{C}_0}  \arg s +
   \Delta_{\mathcal{C}_0}  \arg
  B^{\frac{s}{2}}  
    + r_1  \Delta_{\mathcal{C}_0} \arg \Gamma( \mbox{$\frac{s}{2}$})
  +r_2  \Delta_{\mathcal{C}_0} \arg \Gamma(s)
 \\ + \Delta_{\mathcal{C}_0}  \arg \left( (s-1)\zeta_K(s)\right).  
\end{multline*}
A straightforward calculation yields
\begin{align*} 
 \Delta_{\mathcal{C}_0} \arg  s & = \arctan(2T), \\
    \Delta_{\mathcal{C}_0} \arg
  B^{\frac{s}{2}} & = \frac{T}{2} \log B = \frac{T}{2} \log \Big( \frac{d_K}{\pi^{n_K}2^{2r_2}} \Big).
\end{align*}
To compute $\Delta_{\mathcal{C}_0} \arg \Gamma(s) = \Delta_{\mathcal{C}_0} \left(\Im \log \Gamma(s)\right)$, we use Stirling's formula as given by  Olver \cite[p.\ 294]{Ol}
\begin{equation}
   \log \Gamma(z) = (z-\frac{1}{2}) \log z -z + \frac{\log 2 \pi}{2}
   + \frac{\theta}{6|z|}
   \label{eq:stirling}
\end{equation}
with $|\arg z| \le \frac{\pi}{2}$ and $|\theta| \le 1$.  
It follows, as in p.\ 268 of \cite{Mc}, that 
\begin{align*}
   \Big| \Im \log \Gamma \Big( \mbox{$\frac{1}{4}$}+i\frac{T}2\Big)
  - \frac{T}2 \log \Big(\frac{T}{2e} \Big) \Big| 
&   \le  \frac{T}4 \log\Big(1+\frac{1}{4T^2}\Big) + \frac14  \arctan(2T) + \frac{1}{3\sqrt{\frac{1}{4}+T^2}},
\\
   \Big| \Im \log \Gamma ( \mbox{$\frac{1}{2}$}+iT)
  - T\log \Big(\frac{T}{e}\Big) \Big| 
&   \le \frac{T}{2} \log \Big(1+\frac{1}{4 T^2} \Big) + \frac{1}{6\sqrt{\frac{1}{4}+T^2}}.
\end{align*}
As both functions on the right are decreasing for $T \ge 1$, it follows that 
\begin{align*}
  \Big| \Im \log \Gamma \Big( \mbox{$\frac{1}{4}$}+i\frac{T}2\Big)
  - \frac{T}2 \log \Big(\frac{T}{2e} \Big) \Big| 
&   \le 0.630716,
\\      \Big| \Im \log \Gamma \Big( \mbox{$\frac{1}{2}$}+iT \Big)
  - T\log \Big( \frac{T}{e} \Big) \Big| & \le 0.260643.
\end{align*}
Therefore
\begin{align*}
   \Delta_{\mathcal{C}_0} \arg \Gamma( \mbox{$\frac{s}{2}$}) 
    & =\frac{T}{2}\log \Big(\frac{T}{2e} \Big)
  +0.630716 \theta_1, \\
  \Delta_{\mathcal{C}_0} \arg \Gamma(s)
     & =T\log \Big(\frac{T}{e}\Big) + 0.260643 \theta_2. 
\end{align*}
Combining these facts, we obtain
\begin{multline*}
   \Delta_{\mathcal{C}_0} \arg \xi_K(s)  =      \arctan(2T) + \left( \frac{T}{2} \log(B) + r_1 \frac{T}{2} \log \Big( \frac{T}{2e}\Big)
   +r_2 T \log\Big( \frac{T}{e} \Big) \right)
  \\ +  0.630716 r_1 \theta_1 + 0.260643 r_2 \theta_2 +  \Delta_{\mathcal{C}_0} \arg  \left((s-1) \zeta_K(s)\right).
\end{multline*}
Since $r_1+2r_2=n_K$, we have
\[
  \frac{T}{2} \log(B) + r_1 \frac{T}{2}\log\Big( \frac{T}{2e} \Big)  +r_2 T\log \Big( \frac{T}{e} \Big)
  = \frac{T}{2}  \log \Big( d_K  \Big( \frac{T}{2 \pi e} \Big)^{n_K}   
  \Big).
\]
Combining  \eqref{eq:NKT} with the last two equations yields
\begin{equation}
  \label{eq:NKTid}
  N_K(T) = \frac{T}{\pi} \log \Big( d_K  \Big( \frac{T}{2 \pi e} \Big)^{n_K}   
  \Big) 
+ \frac{2}{\pi} \Delta_{\mathcal{C}_0} \arg \left((s-1) \zeta_K(s)\right)
+ \frac{1.261431 n_K}{\pi}\theta_3 +\theta_4.
\end{equation} 
In order to complete the argument we must bound
$ \frac{2}{\pi} \Delta_{\mathcal{C}_0} \arg \left( (s-1)\zeta_K(s)\right) $.  
We divide $\mathcal{C}_0$
into the contours $\mathcal{C}_1$ and  $\mathcal{C}_2$ as follows:
\begin{align*}
    \mathcal{C}_1:  \sigma_1 \text{ to } \sigma_1+iT
    \text{ and }
   \mathcal{C}_2:  \sigma_1+iT \text{ to } \frac{1}{2}+iT.  
\end{align*}
We begin with the argument change on $\mathcal{C}_1$.  If $\sigma > 1$
then 
\[
   |\arg \zeta_K(s)| \le 
   |\log \zeta_K(s)| \le \log \zeta_{K}(\sigma) \le n_K \log \zeta(\sigma) 
\]
and therefore, since $\sigma_1 = \frac{3}{2}+2\eta$,
\[
  |\Delta_{\mathcal{C}_1} \arg \zeta_K(s)| 
   \le n_K \log  \zeta( \tfrac{3}{2}+2 \eta). 
\]
In addition, 
$\Delta_{\mathcal{C}_1} \arg (s-1) = \arctan ( \frac{T}{\sigma_1-1} )= \arctan ( \frac{T}{2\eta+\frac12} )$ and
we deduce that 
\begin{equation}
  \label{eq:deltaC1}
  \frac{2}{\pi} |\Delta_{\mathcal{C}_1} \arg (s-1)\zeta_K(s)|
\le \frac{2 n_K}{\pi} \log  \zeta( \tfrac{3}{2}+2 \eta) +1. 
\end{equation}
We now bound the argument change on $\mathcal{C}_2$. 
Let $a(w)=(w-1)\zeta_K(w)$ and consider 
\begin{equation}
   \label{eq:fdefn}
   f(w) =\frac{1}{2}( a(w+iT)^N + a(w-iT)^N), \text{ where } N \in \mathbb{N}. 
\end{equation}
Note that 
\[
    f(\sigma) = \Re \, a(\sigma+iT)^N \text{ if } \sigma \in \mathbb{R}.
\]
Suppose $f(\sigma)$ has $n$ real zeros in the interval
$\frac{1}{2} \le \sigma \le \sigma_1$.  These zeros partition the interval
into $n+1$ subintervals.  On each of these subintervals $\arg a(\sigma+iT)^N$
can change by at most $\pi$, since $\Re \, a(\sigma +iT)^N$ is
nonzero on the interior of each subinterval.  It follows that 
\begin{equation}
   \label{eq:deltaC2}
   |\Delta_{\mathcal{C}_2} \arg a(s)|=
   \frac{1}{N} | \Delta_{\mathcal{C}_2} \arg a(s)^N|
   \le \frac{(n+1) \pi}{N}. 
\end{equation}
We now provide an upper bound for $n$.  Let $0 < \eta < \frac{1}{2}$ and  \[\sigma_0=1 +\eta.\]
Jensen's theorem asserts that 
\[
   \log|f(\sigma_0)| + \int_{0}^{1+2\eta} \frac{n(r)dr}{r} =  \frac{1}{2 \pi} \int_{-\frac{\pi}{2}}^{\frac{3 \pi}{2}}
    \log |f(\sigma_0+(1+2\eta) e^{i \theta}| d \theta,
\]
where $n(r)$ denotes the number of zeros of $f(z)$ in the circle centered at $\sigma_0$ of radius $r$. 
Observe that $n(r) \ge n$ for  $r \ge \frac{1}{2}+\eta$ and thus
\begin{equation}
    n \log 2 \le \frac{1}{2 \pi} \int_{-\frac{\pi}{2}}^{\frac{3 \pi}{2}}
    \log |f(\sigma_0+(1+2\eta) e^{i \theta}| d \theta
    - \log|f(\sigma_0)|. 
    \label{eq:jen}
\end{equation}
The next step is to provide an upper bound for the integral. 
Rademacher  proved an explicit version of the Phragm\'{e}n-Lindel\"{o}f Theorem.  
Theorem 4 of \cite{Ra} states that
\begin{equation*}
    |\zeta_K(w)|
     \le  3 \frac{|1+w|}{|1-w|} \mbox{$
     (d_K ( \frac{|w+1|}{2 \pi})^{n_K})^{\frac{1+\eta-\Re (w)}{2}}
     $} \zeta(1+\eta)^{n_K} 
\end{equation*}
uniformly for $-\eta \le \Re(w) \le 1+\eta$.  An examination of the proof reveals that 
the slightly stronger bound
\begin{equation}
   \label{eq:zbd2}
    |\zeta_K(w)|
     \le  3 \frac{|1+w|}{|1-w|} \mbox{$
     (d_K ( \frac{|w+1|}{2 \pi})^{n_K})^{\frac{1+\eta-\Re (w)}{2}}
     $} \zeta_K(1+\eta),
\end{equation}
holds for $-\eta \le \Re(w) \le 1+\eta$. Bound (7.1) in \cite{Ra} is $|\zeta_K(1+\eta+it)| \le \zeta(1+\eta)^{n_K}$
 for $\eta > 0$.  However, this may be replaced by 
\[
    |\zeta_K(1+\eta+it)| \le \zeta_K(1+\eta)
\]
for $\eta >0$ and this change in the argument immediately leads to \eqref{eq:zbd2}. 
It follows that, for $w=  \sigma_0 + (1+2\eta)e^{i \theta} $  with $\theta \in [\frac{\pi}{2}, \frac{3 \pi}{2}]$, 
\begin{equation}
   \label{eq:abd}
    |a(w\pm i T)|
     \le  3 |1+w \pm iT | \mbox{$
     (d_K ( \frac{|1+w\pm iT|}{2 \pi})^{n_K})^{-\frac{1}{2}(1+2\eta) \cos \theta}
     $} \zeta_K(1+\eta) .
\end{equation}
Since $T \ge 1$ and  $0< \eta \le \frac{1}{2}$,
\begin{equation}
  \label{eq:1wiTbd}
  |1+ w \pm iT|  \le  |1+\sigma_0 \pm iT| + 1+2\eta 
 = \sqrt{ T^2+(2+\eta)^2 } +1+ 2\eta
 \le \sqrt{T^2+ (\tfrac{5}{2})^2}+2
\end{equation}
and thus
\begin{equation}
    \label{eq:logabs}
      \log |1+ w \pm iT| 
        \le \log(b_1T),
\end{equation}
where 
\begin{equation}\label{def-C1}
b_1 = \sqrt{(1+(\tfrac{5}{2})^2)}+2=4.692582\ldots.
\end{equation} 
Putting together \eqref{eq:fdefn}, \eqref{eq:abd}, \eqref{eq:1wiTbd}, taking logarithms, and then applying \eqref{eq:logabs} 
gives
\begin{multline*}
   \log|f(\sigma_0 + (1+2 \eta)  e^{i \theta})|  
 \le - 
       \frac{N}{2}(1+2 \eta)(\cos\theta)  
       ( \Lo
       + n_K 
       \log (\tfrac{b_1 T}{2 \pi} ))   \\
 + N 
  \left( \log (3 b_1 T ) 
+ \log \zeta_K(1+\eta)\right),
\end{multline*}
valid for $\theta \in [\frac{\pi}{2}, \frac{3 \pi}{2}]$.
Applying this bound on the left-hand side of the contour in \eqref{eq:jen} and employing the integrals
$-\frac1{2\pi}\int_{\frac{\pi}{2}}^{\frac{3 \pi}{2}} (\cos \theta) \, d \theta =\frac1{\pi}$ 
and $\frac1{2\pi}\int_{\frac{\pi}{2}}^{\frac{3 \pi}{2}} \, d \theta = \frac12$, we find that
\begin{multline}
  \label{eq:left}
  \frac{1}{2 \pi} \int_{\frac{\pi}{2}}^{\frac{3 \pi}{2}}
  \log|f(\sigma_0 + (1+2 \eta)  e^{i \theta})| d \theta 
   \le  \frac{N}{2 \pi} 
       (1+2 \eta)
       ( \Lo 
       + n_K 
       \log (\tfrac{b_1 T}{2 \pi} ) 
      )   
\\  + \frac{N}{2} \log(3 b_1 T)  
+\frac{N}{2} \log \zeta_K(1+\eta). 
\end{multline}
For the right part of the contour in \eqref{eq:jen}, we shall make use of
the  bound
\begin{equation*}
   |f(\sigma_0+(1+2\eta)e^{i\theta})| \le (1+3\eta+T)^{N}\zeta_K(1+\eta)^{N}
\end{equation*}
valid for $\theta\in[-\pi/2,\pi/2]$.  This implies that
\begin{equation}
\label{eq:right}
   \frac{1}{2 \pi} \int_{-\frac{\pi}{2}}^{\frac{\pi}{2}}
    \log |f(\sigma_0+(1+2\eta) e^{i \theta}| d \theta 
    \le \frac{N}{2} \log(1+3 \eta +T)+\frac{N}{2} \log\zeta_K(1+\eta).
\end{equation}
Together with \eqref{eq:jen}, \eqref{eq:left}, and \eqref{eq:right}, we obtain
\begin{multline}\label{eq:nlog2}
  n  \log 2  \le 
  \frac{N}{2 \pi} 
       (1+2 \eta)
       ( \Lo
       + n_K 
       \log (\tfrac{b_1 T}{2 \pi} ) )      + \frac{N}{2}  \log(3 b_1 T) 
 +\frac{N}{2} \log(1+3 \eta +T) \\ + N \log \zeta_K(1+\eta)     - \log|f(1+\eta)|. 
\end{multline} 
To complete our bound for $n$, we require a lower bound for $\log|f(1+\eta)|$.
We write $a(1+\eta+iT)=re^{i \phi}$ and then choose (by Dirichlet's approximation theorem) a sequence
of $N$'s tending to infinity such that $N \phi$ tends to $0$ modulo 
$2 \pi$.   It follows that 
\begin{equation}
   \label{eq:limit}
   \lim_{N \to \infty} \frac{f(1+\eta)}{|a(1+\eta+iT)|^N}
   = 1 . 
\end{equation}
Note that, for $\sigma  > 1$, we have 
\begin{equation}
  \label{eq:zetaKlb}
  |\zeta_K(s)| = \prod_{\mathfrak{p}}|1-N(\mathfrak{p})^{-s}|^{-1}
  \ge \prod_{\mathfrak{p}} (1+ \frac{1}{N(\mathfrak{p})^{\sigma}})^{-1}
  = \frac{\zeta_{K}(2 \sigma)}{\zeta_K(\sigma)} 
\end{equation}
and $|1+\eta + iT-1| = \sqrt{\eta^2+T^2}$ so that
\begin{equation}
  \label{eq:a1etalb}
  |a(1+\eta+iT)| \ge  \sqrt{\eta^2+T^2}  \,  \frac{\zeta_{K}(2 +2\eta)}{\zeta_K(1+\eta)}.
\end{equation}
Thus we derive from \eqref{eq:limit}, \eqref{eq:zetaKlb}, and \eqref{eq:a1etalb} that
\[
\log |f(1+\eta)| \ge N\log\left( \sqrt{\eta^2+T^2}\, \frac{\zeta_K(2+2\eta)}{\zeta_K(1+\eta)} 
\right)+o(1) ,
\]
where the term $o(1)\rightarrow 0$ as $N\rightarrow \infty$. Equation \eqref{eq:nlog2} becomes
\begin{multline*}
  n  \log 2 \le   \frac{N}{2 \pi}        (1+2 \eta)       ( \Lo       + n_K        \log (\tfrac{b_1 T}{2 \pi} )       )    
  + \frac{N}{2} ( \log(3b_1 T) + \log(\tfrac{5}{2}+T) ) 
  \\ + N \log \zeta_K(1+\eta)   
    -  N \log  \frac{\zeta_{K}(2 +2 \eta)}{\zeta_K(1+\eta)}   -N \log\sqrt{\eta^2+T^2} +o(1) .
\end{multline*} 
We combine the third and fourth terms and then use the inequality
$
\frac{\zeta_K(\sigma)^2}{\zeta_K(2 \sigma)}
\le \Big(\frac{\zeta(\sigma)^2}{\zeta(2 \sigma)} \Big)^{n_K}
$
to obtain
\begin{multline*}
  n  \log 2  \le 
  \frac{N}{2 \pi} 
       (1+2 \eta)
       ( \Lo
       + n_K 
       \log (\tfrac{b_1 T}{2 \pi} ) 
      )    
    + \frac{N}{2} ( \log(3b_1 T) + \log(\tfrac{5}{2}+T) )
\\ +  N n_k \log \frac{\zeta (1 +\eta)^2}{\zeta (2+2\eta)}  
   -N \log\sqrt{\eta^2+T^2} + o(1)  .
\end{multline*} 
By the last inequality and by \eqref{eq:deltaC2}, we have 
\begin{multline} \label{eq:deltaC2b} 
 \frac{2}{\pi} |\Delta_{\mathcal{C}_2} \arg a(s) |
   \le 
  \frac{1+2 \eta}{\pi \log 2} 
       \left( \Lo
       + n_K 
       \log \left(\tfrac{b_1 T}{2 \pi} \right) 
      \right)   + \frac{1}{ \log 2} ( \log(3b_1 T) + \log(\tfrac{5}{2}+T) )
 \\+  \frac{2n_K}{\log 2} \log  
   \frac{\zeta(1 +\eta)^2}{\zeta(2+2\eta)}     
   - \frac{\log(\eta^2+T^2)}{\log 2}  + o(1)  
\end{multline} 
where $o(1)\rightarrow 0$ as $N\rightarrow \infty$.
We let $N \to \infty$ and combine the results obtained for $\mathcal{C}_1$ in \eqref{eq:deltaC1} and for $\mathcal{C}_2$ in \eqref{eq:deltaC2b}:
\begin{multline*}
 \frac{2}{\pi} | \Delta_{\mathcal{C}_0} \arg (s-1) \zeta_K(s)  |  
 \le  \frac{1+2 \eta}{\pi \log 2}  \left( \Lo + n_K  \log \left(\tfrac{b_1 T}{2 \pi} \right)  \right)    
  + \frac{ \log(3b_1 T) + \log(\tfrac{5}{2}+T) }{\log 2}
\\ +  \frac{2n_K}{\log 2} \log  \frac{\zeta(1 +\eta)^2}{\zeta(2+2\eta)}  
   - \frac{\log(\eta^2+T^2)}{\log 2} + \frac{2n_K}{\pi} \log  \zeta( \tfrac{3}{2}+2 \eta) +1.
\end{multline*}
Inserting this in \eqref{eq:NKTid} yields
\begin{align*}
  | N_K(T)-\frac{T}{\pi} \log \Big( d_K  \Big( \frac{T}{2 \pi e} \Big)^{n_K}   
  \Big) | &  \le  
  c_1(\eta)(  \Lo  + n_K \log T) + c_2(\eta) n_K + g(T),
\end{align*}
where 
\begin{align}
   c_1(\eta) & =  \frac{1+2 \eta}{\pi \log 2},  \\
   c_2(\eta) & =  \frac{1+2 \eta}{\pi \log 2} \log \Big( \frac{b_1}{2 \pi} \Big)
    + \frac{2}{\log 2} \log  
   \frac{\zeta(1 +\eta)^2}{\zeta(2+2\eta)} 
   + \frac{2}{\pi} \log  \zeta( \tfrac{3}{2}+2 \eta)
+ \frac{1.261431}{\pi}, \\
  g(T) & = \frac{1}{\log 2} \left( \log(3b_1 T) + \log(\tfrac{5}{2}+T) 
   - \log(\eta^2+T^2)\right) + 2.
\end{align}
Observe that 
\[
   c_2(\eta) 
 = b_2-b_3\eta  
    + \frac{2}{\log 2} \log \frac{\zeta(1+\eta)^2}{\zeta(2 +2 \eta)}  
    +\frac{2}{\pi} \log \zeta( \tfrac{3}{2}+ 2 \eta) 
\]
where $b_1$ is defined in \eqref{def-C1}, and
\begin{align}
& \label{def-C2}
  b_2 =  \frac{\log(\tfrac{b_1}{2 \pi})}{\pi (\log 2)} + \frac{1.261431}{\pi} 
  = 0.267481 \ldots, \\
&  \label{def-C3}
b_3  = - \frac{2 \log(\tfrac{b_1}{2 \pi})}{\pi \log 2} = 0.268089  \ldots \ . 
  \end{align}
Since $0 < \eta \le \frac{1}{2}$, we have that, for $T\ge 1$, 
\[
  g(T) 
                  \le \frac{1}{\log 2} \log(1+ \tfrac{5}{2T}) + 2 + \frac{\log(3b_1)}{\log 2}  
                =  7.622699\ldots .
\]
\end{proof}
We now move on to the proof of Theorem \ref{gr}. First, we require a couple of lemmas 
on some real-valued functions. 
For $a, b, c, x \in \mathbb{R}$, define
\begin{equation}
  \label{eq:g}
  g(a,b,c;x) := \kappa \Big( 
  \frac{a}{a^2+x^2} + \frac{b}{b^2+x^2}
  \Big) - \frac{c}{c^2+x^2},
\end{equation}
with $\kappa=\frac1{\sqrt5}$.
\begin{lemma} \label{gbds}
Let $a_0 = \frac{\sqrt{5}-1}{2}, b_0 = \frac{1+\sqrt{5}}{2}$, and $c_0 =1$.  \\
(i) The inequality 
\[
   -0.121585 \ldots \le  g(a_0,b_0,c_0;x) \le 0  
\]
is valid for all $x \in \mathbb{R}$. \\
(ii) Let $0 < \epsilon \le 10^{-2}$ and let $a,b,c \in \mathbb{R}$.  
If $|a-a_0| < 2 \epsilon $, $|b-b_0| <  2 \epsilon$, and $|c-c_0| <  2 \epsilon$,
then 
\[
      -0.121585 \ldots -5 \epsilon \le  g(a,b,c;x) \le 5 \epsilon. 
\]
\end{lemma}
\begin{proof} (i) Differentiating, we find that 
\begin{multline*}
   g'(a_0,b_0,c_0;x)
 \\ = 2x \frac{  -\kappa a_0 (b_0^2+x^2)^2(c_0^2+x^2)^2 - \kappa b_0 (a_0^2+x^2)^2(c_0^2+x^2)^2 + c_0 (a_0^2+x^2)^2(b_0^2+x^2)^2 }{(a_0^2+x^2)^2 (b_0^2+x^2)^2(c_0^2+x^2)^2} .
\end{multline*}
The polynomial in the numerator is of the form $Ax^8 + Bx^6+Cx^4+Dx^2+E$, and it may be checked that $A=D=0,\ B=2,\ C=4$, and $E=-1$.
Observe that the polynomial $2x^6+4x^4-1$ has one positive real root  $\beta = 0.672016 \ldots $.  It follows from calculus that 
\[ 
 0 \ge g(a_0,b_0,c_0;x) \ge g(a_0,b_0,c_0;\beta)= -0.121585 \ldots .
\]
(ii) We begin by considering the difference 
\begin{multline*}
   g(a,b,c;x) - g(a_0,b_0,c_0;x) \\ = 
\kappa \Big( 
  \frac{a}{a^2+x^2}-\frac{a_0}{a_{0}^2+x^2} + \frac{b}{b^2+x^2}-\frac{b_0}{b_{0}^2+x^2}
  \Big) - \frac{c}{c^2+x^2} + \frac{c_0}{c_{0}^2+x^2}. 
\end{multline*}
For real numbers $u$ and $u_0$, we have that 
\[
   \Big| \frac{u}{u^2+x^2} - \frac{u_0}{u_{0}^2 +x^2} \Big|
  =  \Big| \frac{(u-u_0)(x^2-uu_0)}{(u^2+x^2)(u_{0}^2+x^2)} \Big|
  \le \frac{|u-u_0|}{\min(|u|,|u_0|)^2}. 
\]
Using this bound, the triangle inequality implies that  
\begin{multline*}
  |g(a,b,c;x) - g(a_0,b_0,c_0;x)| 
 \le  
\kappa \Big(
 \frac{2\epsilon}{(a_0 -2 \epsilon)^2}
+ \frac{2 \epsilon}{(b_0 -2 \epsilon)^2}
\Big) 
 + \frac{2 \epsilon}{(c_0 - 2 \epsilon)^2} 
 \\ \le 2\epsilon \Big( 
   \frac{ \kappa}{(a_0 -2 \cdot 10^{-2})^2}
+ \frac{ \kappa}{(b_0 -2 \cdot 10^{-2})^2}
 + \frac{1}{(c_0 -2 \cdot 10^{-2})^2} 
\Big)  
 < 5 \epsilon.
\end{multline*}
The above combined with (i) yields (ii).
\end{proof}
Let $s = \sigma + it$ and  $s_1= \sigma_1 + it$  
where $\sigma_1 = \frac{1}{2}(1+\sqrt{1+4 \sigma^2})$.   For $a=0,1$, define the function 
\[
  f_a(\sigma,t) = \frac{1}{2}  \Re
  \left(
  \frac{\Gamma'}{\Gamma} \left( \frac{s+a}{2} \right) - \frac{1}{\sqrt{5}} 
   \frac{\Gamma'}{\Gamma} \left( \frac{s_1+a}{2} \right)
  \right).
\]
In order to abbreviate notation, we set $\psi(z) = \frac{\Gamma'}{\Gamma}(z)$. 
We shall prove
\begin{lemma}
Let $\epsilon>0$,  $\sigma \in [1,1+\epsilon]$, and $|t|\in [0,1]$. Then
\[
f_a(\sigma,t)  \le  C_{a}(\epsilon),
\]
where 
\begin{equation}\label{eq:Caepsilon}
 C_{a}(\epsilon) =   f_a(1,1) + \epsilon \Big(\frac14 S\Big(\frac{1+a}2,\frac{1}2\Big) +  \frac{1}{2\sqrt{5}}\frac{1+\epsilon}{ \sqrt{1+4(1+\epsilon)^2}} S\Big(\frac{\frac{1+\sqrt5}2+a}2,\frac{1}2\Big) \Big)
\end{equation}
and
\begin{equation}\label{def-S}
S(x,y)=\sum_{n=0}^{\infty} \frac1{(x+n)^2+y^2}  .
\end{equation}
\end{lemma}
Note that for $\epsilon = 0.15$, a Maple calculation gives
\[
  f_a(\sigma, t) \le -0.088955  \text{ for } 1 \le \sigma \le 1.15
\]
which improves McCurley's bound \cite{Mc2} $f_a(\sigma, t) \le -0.0390$ in the same range.  
On the other hand, if we take $\epsilon =0.01$, we obtain 
\begin{equation}
   \label{eq:fsigmatbound}
  \begin{cases}
f_0(\sigma,t)  \le  C_{0}(0.01) <-0.303931,\\ 
f_1(\sigma,t)  \le  C_{1}(0.01) <-0.153758 ,
  \end{cases}
 \text{ for } 1 \le \sigma \le 1.01. 
\end{equation}
\begin{proof}
The lemma shall be proved as follows.  First observe that $f_a(\sigma,t)$ is even in $t$ and thus we may assume that $t \ge 0$. 
The first step is to show that, for fixed $\sigma$,
$f_a(\sigma,t)$ increases with $t$ and thus $f_a(\sigma, t) \le f_a(\sigma, 1)$. 
Then we use the mean value theorem to write $f_a(\sigma,1)=f_a(1,1) +(\sigma-1) 
\cdot  \left. \frac{\partial}{\partial \sigma} f_a(\sigma,1) \right|_{\sigma=\theta}$ for $\theta \in [1,1+\epsilon]$. 
These combine to give the bound 
\begin{equation}
\label{eq-MVT} 
f_a(\sigma, t) \le f_a(1,1) +\epsilon \max_{1 \le \sigma \le 1+\epsilon}  \Big|  \frac{\partial}{\partial \sigma} f_a(\sigma,1) \Big|.
\end{equation}
A Maple calculation gives that $f_a(1,1) = -0.312948 \ldots $ if $a=0$, and $=-0.158361 \ldots $ if $ a=1 $.
The final step is to establish a bound for the above maximum term.
We now show that $f_a(\sigma, t)$ increases with $t$. 
We have the identity
\[
  \frac{\Gamma'}{\Gamma}(x+iy) = - \gamma - \frac{1}{x+iy} + 
  \sum_{n=1}^{\infty} \Big( \frac{1}{n} - \frac{1}{n+x+iy} \Big) \text{ for } x \ge 0. 
\]
Setting 
\[
   x_1 = x_1(\sigma)= \frac{\sigma+a}{2},\  x_2 = x_2(\sigma)=\frac{\sigma_1+a}{2}, \text{ and } y = \frac{t}{2},
\]
it follows that 
\begin{align*}
  2f_a(\sigma,t) & =  -\gamma \Big(1-\frac{1}{\sqrt{5}} \Big) + g_0(\sigma,y) + \sum_{n=1}^{\infty} g_n(\sigma,y)
\end{align*}
where 
\begin{align*}
 g_0(\sigma,y) & = - \frac{x_1}{x_1^2 + y^2} + \frac{1}{\sqrt{5}} \frac{x_2}{x_2^2+y^2}, \\
 g_n(\sigma,y) & = \frac{1}{n} \Big(1-\frac{1}{\sqrt{5}} \Big)
- \frac{n+x_1}{(n+x_1)^2+y^2}+ \frac{1}{\sqrt{5}} \frac{n+x_2}{(n+x_2)^2+y^2} \text{ for } n \ge 1. 
\end{align*}
It suffices to prove that $g_n(\sigma, y)$ is increasing in $y$
for $n \in \mathbb{N} \cup \{ 0 \}$.
We have that 
\begin{align*}
 \frac{\partial}{\partial y} g_{n}(\sigma, y) & = 2y \Big( 
 \frac{(n+x_1)}{((n+x_1)^2 +y^2)^2}- \frac{(n+x_2)}{\sqrt{5} ((n+x_2)^2+y^2 )^2}
 \Big) .
\end{align*}
Simplifying, we find that  $\frac{\partial}{\partial y} g_n(\sigma, y) =  \frac{2yQ_n(y)}{\sqrt{5}((n+x_1)^2 +y^2)^2((n+x_2)^2 +y^2)^2}$,
where 
\[
  Q_n(y) = \sqrt{5}(n+x_1)((n+x_2)^2 +y^2)^2 - (n+x_2)((n+x_1)^2 +y^2)^2.
\]
It suffices to prove that $Q_n(y)$ is positive. 
Observe that $Q_n(y)=Ay^4+By^2+C$ where
\begin{align*}
  A & = \sqrt{5}(n+x_1)-(n+x_2),\\
  B & = 2(n+x_1)(n+x_2)\Big( \sqrt{5}(n+x_2) - (n+x_1) \Big), \\
  C & = (n+x_1)(n+x_2) \Big( \sqrt{5}(n+x_2)^3-(n+x_1)^3 \Big). 
\end{align*}
It is easy to see that $A$ is positive if   $\sqrt{5}x_1-x_2$ is positive,
$B$ is positive if  $\sqrt{5}x_2-x_1$ is positive, and $C$ is positive if  
$5^{\frac{1}{6}}x_2-x_1$ is positive. 
Note that
\[
  \sqrt{5}x_1-x_2 = \sqrt{5} \Big( \frac{\sigma+a}{2} \Big) - \frac{\sigma_1+a}{2}
  = \frac{\sqrt{5} \sigma - \sigma_1}{2}+ \frac{a}{2}(\sqrt{5}-1) >0
\]
since $\sqrt{5} \sigma > \sigma_1$. 
Since $x_1 < x_2$ for $\sigma > 1$, it follows that $5^{1/j}x_2 - x_1 >0$ for $j \in \mathbb{N}$. 
Thus $A,B,$ and $C$ are positive and it follows that 
$\frac{\partial}{\partial t} f_a(\sigma,t) > 0$ as desired. 

In order to finish the proof, we will bound $ \frac{\partial}{\partial \sigma} f_a(\sigma,1)$. Observe that 
\[
 f_a(\sigma,1) = \frac12 \Re\Big( \psi\Big(\frac{\sigma+a+i}2\Big) -\frac1{\sqrt5} \psi\Big(\frac{\frac{1+\sqrt{1+4\sigma^2}}2+a+i}2\Big) \Big).
\]
Thus
\[
\frac{\partial f_a(\sigma,1)}{\partial{\sigma}}
= 
 \frac12 \Re\Big( \frac12 \psi' \Big(\frac{\sigma+a+i}2\Big) -  \frac{\sigma}{ \sqrt{5}  \sqrt{1+4\sigma^2}}\psi' \Big(\frac{\frac{1+\sqrt{1+4\sigma^2}}2+a+i}2\Big) \Big).
\]
It is well known that $  \psi'(z) = \sum_{n=0}^{\infty} (n+z)^{-2}$ and thus  
$
\left| \psi'(z) \right| \le S(\Re(z),\Im(z)),
$
where $S$ is given in \eqref{def-S} and decreases with $\Re (z)$.
Since  $\frac{\sigma}{\sqrt{1+4\sigma^2}}$ increases with $\sigma$,
we have
\[
\Big| \frac{\partial f_a(\sigma,1)}{\partial{\sigma}} \Big|
 \le 
 \frac14 S\Big(\frac{1+a}2,1/2\Big) +  \frac{1}{2\sqrt{5}}\frac{1+\epsilon}{ \sqrt{1+4(1+\epsilon)^2}} S\Big(\frac{\frac{1+\sqrt5}2+a}2,1/2\Big).
\]
We combine the above together with \eqref{eq-MVT} to obtain the bound $C_{a}(\epsilon)$ as given by \eqref{eq:Caepsilon}.
\end{proof}
With the previous lemmas in hand, we may now prove Theorem \ref{gr}. 
\begin{proof}[Proof of Theorem \ref{gr}]
Let $s=\sigma +it$, $0<\epsilon \le 10^{-2}$, and assume $1 < \sigma \le 1+\epsilon$.   
Recall that $\mathcal{R}_{\epsilon,t}$ was defined to be the set of $\varrho$ which satisfy 
$1-\epsilon \le \Re(\varrho) <1$ and $|\Im(\varrho) - t| \le 1$. 
Recall that $\xi_K(s)$ is defined by (\ref{def-xi}) and it is 
entire and of order one.  By the Hadamard-Weierstrass factorization theorem 
\[
  \xi_{K}(s) = e^{A+Bs} \prod_{\varrho} \Big(1-\frac{s}{\varrho} \Big) e^{\frac{s}{\varrho}}
\]
where $\varrho$ ranges through the non-trivial zeros of $\zeta_K(s)$. Logarithmically differentiating this expression and employing $\xi_K(s)=\xi_K(1-s)$
leads to the global formula
\begin{equation}\label{explicit-formula}
-\Re \frac{\zeta_K'}{\zeta_K}(s) = - \sum_{\varrho} \Re 
 \Big( \frac1{s-\varrho} \Big) + \frac12\Lo + \Re \Big( \frac1s \Big) + \Re  \Big( \frac1{s-1}  \Big) + \Re \frac{\gamma'_K}{\gamma_K}(s).
\end{equation}
For full details of the derivation see equation (5.9) and Lemma 5.1 of  \cite{LO} or \cite[p.\ 965]{Sta}.
We now employ Stechkin's differencing method.  
Let $\kappa = \frac{1}{\sqrt{5}}$, 
$\sigma_1=\frac{1}{2}\left(\sqrt{1+4 \sigma^2}+1\right)$, and $s_1 = \sigma_1 + it$.  Consider the difference
\begin{multline}\label{Stechkin-explicit-formula}
-\Re \left( \frac{\zeta_K'}{\zeta_K}(s)-\kappa \frac{\zeta_K'}{\zeta_K}(s_{1})\right) =  
-\sum_{\varrho} \Re \left(\frac{1}{s-\varrho}-\kappa\frac1{s_{1}-\varrho}\right)  
+ \frac{1-\kappa}2\Lo
\\+\Re\left(\frac{1}{s} + \frac{1}{s-1}-\frac{\kappa}{s_{1}} - \frac{\kappa}{s_{1}-1}\right)  + \Re\left( \frac{\gamma'_K}{\gamma_K}(s) -\kappa \frac{\gamma'_K}{\gamma_K}(s_{1}) \right).
\end{multline}
The benefit of this formula is that the coefficient of $\Lo$ has been reduced from $\frac{1}{2}$ to  $\frac{1-\kappa}{2}$. 
This argument allowed Stechkin to obtain an improved explicit zero-free region for $\zeta(s)$. 
Next observe that if $\varrho$ is a zero of $\zeta_K(s)$ then so is $1-\overline{\varrho}$. 
It follows that for $w \in \mathbb{C}$
\[
    \sum_{\varrho} \Re \Big( \frac{1}{w-\varrho}  \Big) =
    \sideset{}{'} \sum_{\beta \ge \frac{1}{2}}  
    \Re \Big(
    \frac{1}{w-\varrho} + \frac{1}{w-1+\overline{\varrho}}
    \Big),
\] 
where $\sideset{}{'} \sum_{\beta \ge \frac{1}{2}} $ means that those terms with $\beta=\frac{1}{2}$ are counted with 
weight one-half. 
Define, for complex $s$ and $z$, 
\[
   D(s,z) = \Re (  (s-z)^{-1} + (s-1+\overline{z})^{-1}).
\]
Therefore 
\begin{multline}\label{Stechkin-explicit-formulab}
-\Re \left( \frac{\zeta_K'}{\zeta_K}(s)-\kappa \frac{\zeta_K'}{\zeta_K}(s_{1})\right) =  
-\sideset{}{'} \sum_{\beta \ge \frac{1}{2}}   (D(s,\varrho)-\kappa D(s_1,\varrho))
+ \frac{1-\kappa}{2}\Lo
\\+\left( D(s,1) - \kappa D(s_1,1)\right)  
 + \Re\left( \frac{\gamma'_K}{\gamma_K}(s) -\kappa \frac{\gamma'_K}{\gamma_K}(s_{1}) \right).
\end{multline}
Stechkin (\cite{St}, Lemma 2) proved that if $\Re(s) > 1$ and $\frac{1}{2} \le \Re (z) \le 1$, then 
\[
  D(s,z) - \kappa D(s_1,z) \ge 0.
\]
Applying this positivity result, we discard those $\varrho \notin \mathcal{R}_{\epsilon,t}$ to obtain
\begin{multline}
 \label{eq:diffid}
-\Re \left( \frac{\zeta_K'}{\zeta_K}(s)-\kappa \frac{\zeta_K'}{\zeta_K}(s_{1})\right)  \le   
-  \sum_{\varrho \in \mathcal{R}_{\epsilon,t}}
 (D(s,\varrho)-\kappa D(s_1,\varrho))
+ \frac{1-\kappa}2\Lo
\\
+\left( D(s,1) - \kappa D(s_1,1)\right)   + \Re\left( \frac{\gamma'_K}{\gamma_K}(s) -\kappa \frac{\gamma'_K}{\gamma_K}(s_{1}) \right).
\end{multline}
Observe that 
\[
-\big(  D(s,\varrho)-\kappa D(s_1,\varrho) \big)
 = -  
\Re \Big( \frac{1}{s-\varrho} \Big) 
  + 
\Re\Big(
\frac{\kappa}{s_1-\varrho} 
+ \frac{\kappa}{s_1-1 +\overline{\varrho}}  - \frac{1}{s-1+\overline{\varrho}} \Big).
\]
We shall prove that the conditions on $\sigma, t, \beta,$ and $\gamma$ imply that 
\[
\Re\Big(
\frac{\kappa}{s_1-\varrho} 
+ \frac{\kappa}{s_1-1 +\overline{\varrho}}  - \frac{1}{s-1+\overline{\varrho}} \Big)
\le  5 \epsilon.
\]
The expression we want to bound equals
\begin{align*}
  & \kappa \frac{\sigma_1-\beta}{(\sigma_1-\beta)^2+(t-\gamma)^2} + \kappa \frac{\sigma_1-1+\beta}{(\sigma_1-1+\beta)^2+(t-\gamma)^2}
    - \frac{\sigma-1+\beta}{(\sigma-1+\beta)^2+(t-\gamma)^2}  .
\end{align*}
Note that this is of the form $g(a,b,c;x)$, as defined in \eqref{eq:g}, where 
$a = \sigma_1-\beta$,  $b = \sigma_1-1+\beta$, $c=\sigma-1+\beta$, and $x=t-\gamma$.
The assumption  $1 < \sigma \le 1+\epsilon$ with $0<\epsilon \le 0.01$ implies that
 \[
   \frac{1+\sqrt{5}}{2} \le \sigma_1 \le \frac{1+\sqrt{5}}{2} +\epsilon .
\]
Recall that  $1-\epsilon \le \beta < 1$. 
Together these conditions imply that 
\[
  |a-a_0| \le 2 \epsilon, |b-b_0| \le 2 \epsilon, \text{ and } |c-c_0| \le 2\epsilon
\]
where $a_0, b_0,$ and $c_0$ are the constants defined in Lemma 5. 
By Lemma \ref{gbds} (ii), we have 
$g(a,b,c;x) \le 5 \epsilon$. 
Thus 
\begin{equation}
  \label{eq:sumDineq}
  -  \sum_{\varrho \in \mathcal{R}_{\epsilon,t}}
 (D(s,\varrho)-\kappa D(s_1,\varrho))
 \le   -\sum_{\varrho \in \mathcal{R}_{\epsilon,t}} \Re \Big(  \frac{1}{s-\rho} \Big)
 + 5 \epsilon |\mathcal{R}_{\epsilon,t}|. 
\end{equation}
Next, we observe that
\[
 D(s,1) - \kappa D(s_1,1)
 = \Re\left(\frac{1}{s-1} \right) - g(a,b,c;t) 
\]
where $a = \sigma_1-1, b=\sigma_1, \text{ and } c = \sigma$.
Since $1 < \sigma \le 1+\epsilon$, it follows that $|a-a_0| \le \epsilon$, $|b-b_0| \le \epsilon$, and $|c-c_0| < \epsilon$.
Therefore by Lemma \ref{gbds} (ii),  $-g(a,b,c;t) \le  0.121586 +5 \epsilon$, and then
\begin{equation}
  \label{eq:Ds1ineq}
 D(s,1) - \kappa D(s_1,1)  \le  \Re\Big(\frac{1}{s-1}  \Big)
+ 0.121586 +5 \epsilon.
\end{equation}
Finally the gamma factors are dealt with.
By the duplication formula 
$\Gamma(s)=	\frac{2^{s-1}}{\sqrt{\pi}}\Gamma(\frac{s}{2})\Gamma(\frac{s+1}{2})$, 
it follows from \eqref{eq:gamK} that
$
\gamma_K(s) =  (2\sqrt{\pi})^{-r_2} \pi^{-n_K s/2}  \Gamma\left(\tfrac{s}{2}\right)^{r_1+r_2}  \Gamma\left(\tfrac{s+1}{2}\right)^{r_2}.
$
Therefore
\begin{multline*}
   \Re   \left( \frac{\gamma'_K}{\gamma_K}(s) -\kappa \frac{\gamma'_K}{\gamma_K}(s_{1}) \right) 
   = -n_K (1-\kappa) \frac{\log \pi}{2}   \\
    + \frac{r_1+r_2}{2} \Re 
   \Big(
   \frac{\Gamma'}{\Gamma} \Big(\frac{s}{2} \Big) -  \kappa \frac{\Gamma'}{\Gamma} \Big(\frac{s_1}{2} \Big)
   \Big)
   + \frac{r_2}{2} \Re 
   \Big(
    \frac{\Gamma'}{\Gamma} \Big(\frac{s+1}{2}\Big) - \kappa \frac{\Gamma'}{\Gamma} \Big(\frac{s_1+1}{2} \Big)
   \Big).
\end{multline*}
By \eqref{eq:fsigmatbound},  the last two terms in the previous equation are bounded by 
\begin{align*}
 (r_1+r_2)C_0(0.01) + r_2C_1(0.01) 
 & \le n_K \max\left(C_0(0.01) , 0.5  (C_0(0.01)+C_1(0.01)) \right)  \\
& < -0.228844 n_K 
\end{align*}
for $1 < \Re(s) \le 1.01$ and $|\Im(s)| \le 1$. 
Thus
\begin{equation}
   \label{eq:gamKdiff}
   \Re   \left( \frac{\gamma'_K}{\gamma_K}(s) -\kappa \frac{\gamma'_K}{\gamma_K}(s_{1}) \right) 
   \le  n_K (  
   -(1-\kappa) \frac{\log \pi}{2} - 0.228844 )
<-0.545240n_K.
\end{equation}
We also need
\begin{equation}
  \label{eq:zetaKs1}
  - \kappa \Re \frac{\zeta_{K}'}{\zeta_K}(s_1)
\le \kappa \Big| \frac{\zeta_{K}'}{\zeta_K}(s_1) \Big|
  \le  - \kappa n_K  \frac{\zeta'}{\zeta}(\sigma_1)
  \le   -   \frac{n_K}{\sqrt{5}}  \frac{\zeta'}{\zeta} \Big(\frac{1+\sqrt{5}}{2} \Big)
 \le 0.509786 n_K.
\end{equation}
By \eqref{eq:diffid} combined with \eqref{eq:sumDineq}, \eqref{eq:Ds1ineq}, \eqref{eq:gamKdiff}, and \eqref{eq:zetaKs1}, 
we arrive at
\begin{multline*}
-\Re \left( \frac{\zeta_K'}{\zeta_K}(s) \right) \le   
\Re\Big( \frac{1}{s-1}\Big) 
-  \sum_{\varrho \in \mathcal{R}_{\epsilon,t}}  \Re \Big( \frac{1}{s -\varrho} \Big)
+ \frac{1-\kappa}2\Lo
 -0.035454 n_K \\ + 5 \epsilon ( |\mathcal{R}_{\epsilon,t}| +1)+ 0.121586. 
\end{multline*}
\end{proof}

\section{An explicit formula}

Classical theorems concerning zero-free regions deal with the logarithmic derivative 
\begin{equation}
\label{eq-log-der-zeta}
   - \frac{\zeta_{K}'}{\zeta_K}(s) =  \sum_{\mathfrak{a} \subset  \mathcal{O}_K  }
    \frac{\Lambda(\mathfrak{a})}{(N\mathfrak{a})^s}
\end{equation}
where $\mathfrak{a}$ ranges through non-zero ideals of $\mathcal{O}_K$ and 
\begin{equation*}
   \label{eq:lambda}
   \Lambda(\mathfrak{a}) =  \begin{cases}
   \log(N \mathfrak{p})     & \text{ if } \mathfrak{a}=\mathfrak{p}^m \text{ where } \mathfrak{p} \text{ is prime}, \\
    0 & \text{ else}. \\
    \end{cases}
\end{equation*}
The best explicit zero-free region theorems  as in \cite{HB} deal with smoothed versions of the logarithmic derivative. 
Throughout this section we set $s=\sigma+it$, $|t| \le 1$,  and 
\begin{equation} 
   \label{eq:Gs}
   \mathcal{K}(s) = 
   \Re \Big(  \sum_{\mathfrak{a} \subset  \mathcal{O}_K  }
    \frac{\Lambda(\mathfrak{a})}{(N\mathfrak{a})^s} f(\Lo^{-1} \log (N \mathfrak{a})) \Big)
\end{equation}
where $\mathfrak{a}$ ranges through non-zero ideals and $f$ is a real-valued function.  In order to derive nice properties of $\mathcal{K}(s)$,
conditions will be imposed on $f$. In this section, we follow very closely the work of Heath-Brown \cite[pp.\ 280-283]{HB}.

\noindent {\it Condition 1}. Let $x_0$ be a positive constant. Let $f$ be  a continuous function from $[0,\infty)$ to 
$\mathbb{R}$ which is supported in $[0,x_0)$ and satisfies $f(0) \ge 0$. 
In addition, $f$ is twice differentiable on $(0,x_0)$, $f''$ is continuous, 
and there exists a positive constant $B=B(f)$ with 
$|f''(t)| \le B(f)$  for  $t \in (0,x_0)$.  

\noindent Associated to $f$ is its Laplace transform
\[
    F(z) = \int_{0}^{\infty} e^{-zt} f(t) \, dt. 
\]
As $f$ is of compact support, $F$ is entire.
We now consider how $F$ decays as $|z| \to \infty$. 
Note that for $\Re(z) >0$, an integration by parts yields
\begin{equation}
   \label{eq:Fasym}
   F(z) = \frac{f(0)}{z} + F_0(z) 
\end{equation}
where 
\begin{equation}
  \label{eq:Fodef}
  F_0(z) = \frac{1}{z} \int_{0}^{x_0} e^{-zt} f'(t) dt. 
\end{equation}
By analytic continuation, \eqref{eq:Fasym} holds for all $z\in \C$. 
Our next step is to determine how $F_0$ decays as $|z| \to \infty$.  
This will require some bounds for $f$ and $f'$. 
By the mean value theorem, there exists $t_0 \in (0,x_0)$ such that 
\[
   |f'(t_0)| = |(f(x_0)-f(0)) x_{0}^{-1}| = |f(0)| x_{0}^{-1}.
\] 
Another application of the mean value theorem yields
\begin{equation}
 \label{eq:cond2}
  |f'(t)| \le |f'(t)-f'(t_0)| + |f'(t_0)| 
   \le B(f)x_0 + |f(0)|x_0^{-1}
\end{equation}
for all $t \in (0,x_0)$.  Likewise, 
\begin{equation}
\begin{split} \label{eq:cond3}
  |f(t)|=|f(t)-f(x_0)| 
  \le |t-x_0| (
  B(f)x_0 + |f(0)|x_{0}^{-1}
  )
  \le B(f)x_{0}^2 + |f(0)|
\end{split}
\end{equation}
for all $t \in (0,x_0)$. 
Integrating \eqref{eq:Fodef} by parts again, it follows that 
\[
  F_0(z) = z^{-2}\big( f'(0^+)-f'(x_0^{-}) e^{-zx_0}\big) + z^{-2} \int_{0}^{x_0} e^{-zt} f''(t)dt
\]
and thus 
\begin{equation}
   \label{eq:F0bd}
   |F_0(z)| \le c(f)|z|^{-2} 
\end{equation}
where 
\begin{equation}
  \label{eq:cf}
  c(f)   =  3B(f)x_0+2|f(0)|x_{0}^{-1}. 
\end{equation}
With this bound for $F_0(z)$, we are prepared to derive an explicit formula and an inequality relating $\mathcal{K}(s)$ to the zeros of $\zeta_K(s)$. 

Let $c > 1$ and suppose that $\sigma > c$.  Consider the contour integral
\begin{align*}
   I & = \frac{1}{2 \pi i} \int_{(c)} \Big(
     - \frac{\zeta_{K}'}{\zeta_K}(w)  \Big) F_0 ((s-w) \Lo) \, dw
\end{align*}
where $\int_{(c)} =  \lim_{T \to \infty} \int_{c-i T}^{c+i T}$. 
Expanding out the Dirichlet series using \eqref{eq-log-der-zeta}, we have
\begin{equation}\label{eq-defI}
I  =  \sum_{\mathfrak{a} \subset  \mathcal{O}_K }
   \Lambda(\mathfrak{a})
      \frac{1}{2 \pi i} \int_{(c)} (N \mathfrak{a})^{-w}
      F_0 ((s-w) \Lo) dw. 
\end{equation}
The Laplace inversion formula is
\[
f(\Lo^{-1} \log x ) =  \frac{1}{2 \pi i} \int_{((\sigma-c) \Lo)} e^{z \tfrac{\log x}{\Lo}} F (z) dz.
\]
We may assume that $x = N \mathfrak{a} > 1$ since $\Lambda(\mathfrak{a})$ is supported on prime powers.  
By the variable change $z=(s-w) \Lo$ and \eqref{eq:Fasym}, we have 
\begin{equation}
\begin{split}
\label{eq-inverseLaplace}
f(\Lo^{-1} \log x) &= \Lo \frac{1}{2 \pi i} \int_{(c)} x^{s-w } F ((s-w) \Lo) dw
\\&= \frac{f(0)}{2 \pi i} \int_{(c)} x^{s-w } \frac{dw}{s-w }  +\frac{\Lo}{2 \pi i} \int_{(c)} x^{s-w } F_0 ((s-w) \Lo) dw.
\end{split}
\end{equation}
It follows from Perron's formula that $\frac{1}{2 \pi i} \int_{(c)} x^{s-w} \frac{dw}{s-w} =
\frac{1}{2 \pi i} \int_{(\sigma-c)} x^z \frac{dz}{z} =1$ since $x>1$ and $\sigma-c>0$. 
Thus \eqref{eq-inverseLaplace} gives
\[
     \frac{1}{2 \pi i} \int_{(c)} x^{-w} F_0 ((s-w) \Lo) dw
     = \Lo^{-1} x^{-s} (f(\Lo^{-1} \log x)-f(0))
\] 
and, together with \eqref{eq-defI}, we obtain
\begin{equation}
  I = \Lo^{-1} \sum_{\mathfrak{a} \subset  \mathcal{O}_K  }
    \frac{\Lambda(\mathfrak{a})}{(N\mathfrak{a})^s} f(\Lo^{-1} \log N \mathfrak{a})  + \Lo^{-1} f(0) \frac{\zeta_{K}'}{\zeta_K}(s) .
  \label{eq:I}
\end{equation}
Next, $I$ is evaluated in a different way.  The line of integration is moved left to
$\Re(w)=- \frac{1}{2}$. The poles of $-\frac{\zeta_{K}'}{\zeta_K}(w) $ are located at $w=1$, with residue $1$, 
at the non-trivial zeros $\varrho$ of $\zeta_K$, with residue $-1$, 
and at $w=0$, since $\zeta_K$ has a trivial zero there, with residue $-(r_1+r_2-1)$.
This yields
\begin{multline}\label{eq:defI}
   I = \frac{1}{2 \pi i} \int_{(-1/2)} \Big( 
     - \frac{\zeta_{K}'}{\zeta_K}(w) 
   \Big) F_0((s-w) \Lo) dw +
  F_0 ((s-1) \Lo) \\ - \sum_{\varrho} F_0 ((s-\varrho) \Lo) - (r_1+r_2-1)F_0(s\Lo).
\end{multline}
By \eqref{eq:F0bd}, 
$\ds{
|F_0(s\Lo) | \le c(f) (|s|\Lo)^{-2} \ll c(f) \Lo^{-2} 
}$.
Thus
\begin{equation}\label{eq:pole0bd}
- (r_1+r_2-1)F_0(s\Lo) \ll c(f) n_K \Lo^{-2} .
\end{equation}
By \eqref{def-xi} and  the functional equation \eqref{eq:xife}, 
\[
   \frac{\zeta_{K}'}{\zeta_K}(w) =
-\frac{\zeta_{K}'}{\zeta_K}(1-w)   - \Lo
   - \frac{\gamma_{K}'}{\gamma_{K}}(w)
   - \frac{\gamma_{K}'}{\gamma_{K}}(1-w).
\]
We have  $|\frac{\zeta_{K}'}{\zeta_K}(1-w)| 
\le - n_K \frac{\zeta'}{\zeta}(\frac{3}{2})$
and, by \eqref{eq:gamK} and Stirling's formula,
\[
   \frac{\gamma_{K}'}{\gamma_K}(z) \ll n_K \log(|z|+2)
\]
for $\Re z = -\frac{1}{2}$ or $\frac{3}{2}$. Therefore 
\[
  \frac{\zeta_{K}'}{\zeta_K}(w)= - \Lo + \mathcal{O}(n_K \log(|w|+2)).
\]
Together with \eqref{eq:F0bd}, we deduce that 
\begin{multline*}
\frac{1}{2 \pi i} \int_{(-1/2)} \Big( - \frac{\zeta_{K}'}{\zeta_K}(w) \Big) F_0((s-w) \Lo) dw 
\\ =  \frac{\Lo}{2 \pi i} \int_{(-1/2)} F_0((s-w) \Lo) dw
+ \mathcal{O}\left( \frac{c(f) n_K}{\Lo^2} 
    \int_{(-1/2)} \frac{\log(|w|+2)}{|s-w|^2} |dw|
   \right).
\end{multline*}
By moving the contour far to the left, it follows from $F_0(z)$ being analytic for $\Re (z)>0$, that 
the first integral is zero. 
The second integral is bounded by $\ds{\mathcal{O}(\log (|s|+2))}$.
By \eqref{eq:defI}, \eqref{eq:pole0bd}, and the fact that $n_K \ll \Lo$, it follows that
\[
   I =  F_0 ((s-1) \Lo) - \sum_{\varrho} F_0((s-\varrho) \Lo) 
   + \mathcal{O}(c(f)  \Lo^{-1} \log (|s|+2))
\]
and thus, with \eqref{eq:I},
\begin{multline}
\label{eq:I-bis}
   \sum_{\mathfrak{a} \subset  \mathcal{O}_K }
\frac{\Lambda(\mathfrak{a})
 }{(N\mathfrak{a})^s} f(\Lo^{-1} \log N \mathfrak{a})
= \Lo F_0 ((s-1) \Lo) 
- \Lo \sum_{\varrho} F_0((s-\varrho)  \Lo) 
  -  f(0)  \frac{\zeta_{K}'}{\zeta_K}(s) 
\\   + \mathcal{O}(c(f) \log (|s|+2)),
\end{multline}
for all $\sigma > c$.

Let $\epsilon>0$ and $|t|\le 1$.
Consider the set $\mathcal{R}_{\epsilon,t}$ defined by \eqref{eq:defRepsilont}.
Those zeros which satisfy $\varrho \notin \mathcal{R}_{\epsilon,t}$
may be discarded with an error
\[
 \ll \Lo \sum_{\varrho \notin \mathcal{R}_{\epsilon,t} } \frac{c(f)}{\Lo^2|s-\varrho|^2}
 \ll_{\epsilon} \Lo^{-1} c(f) \sum_{\varrho} \frac{1}{1+|t-\gamma|^2}.
\]
Assume that $\gamma$ satisfies $k \le |t - \gamma| \le k+1$ where $k \ge 1$. 
Observe that since $|t|\le 1$, then 
\[
k-1 \le  |t-\gamma| - |t| 
  \le  |\gamma| \le |t-\gamma| + |t| \le k+2.
\]
It follows from Theorem \ref{nt} that
$\ds{
\sum_{ k\le |t-\gamma| \le k+1}  1 \le \sum_{ k -1 \le |\gamma| \le k+2}  1 \ll \Lo \log(k+2) ,
}$
and in addition, we have 
$\ds{
\sum_{  |t-\gamma| \le 1}  1 \ll \Lo 
}$.
Employing these bounds, we obtain
\[
   \sum_{\varrho}  \frac{1}{1+|t-\gamma|^2} 
 \le   \sum_{k=0}^{\infty}   \frac{1}{1+k^2}  \sum_{ k\le |t-\gamma| \le k+1} 1
  \ll \Lo \sum_{k=0}^{\infty} \frac{\log(k+2)}{1+k^2}  \ll \Lo ,
\]
giving that the discarded zeros are bounded by $\mathcal{O}(c(f) )$. 
Taking real parts in \eqref{eq:I-bis} and recalling \eqref{eq:Fasym}, we obtain
\begin{multline*}
 \mathcal{K}(s)
   = \Lo \Re F ((s-1) \Lo) - \Re \Big( \frac{f(0)}{s-1} \Big)  - \sum_{\varrho \in \mathcal{R}_{\epsilon,t}}
   \Big(
   \Lo  \Re( F((s-\varrho) \Lo))  
   - \Re \Big( \frac{f(0)}{s-\varrho}  \Big) 
   \Big) \\
   -f(0) \Re \Big(\frac{\zeta_{K}'}{\zeta_K}(s)  \Big)
  + \mathcal{O} ( c(f) ). 
 \end{multline*}
By Corollary \ref{gr2},
\begin{equation*}
    -\Re \left( \frac{\zeta_K'}{\zeta_K}(s) \right)   \le   
\Re\Big( \frac{1}{s-1}\Big) 
-  \sum_{\varrho \in \mathcal{R}_{\epsilon,t}}  \Re \Big( \frac{1}{s -\varrho} \Big)
+ (\phi+ C_2 \epsilon) \Lo
\end{equation*}
for $d_K$ sufficiently large. 
Since $f(0) \ge 1$,  the last two formulae combine to give 
\begin{equation}
  \label{eq:Kineq}
 \mathcal{K}(s)
 \le \Lo \Re F ((s-1) \Lo) 
- \Lo  \sum_{\varrho \in \mathcal{R}_{\epsilon,t}} \Re( F((s-\varrho) \Lo)) 
  + f(0)(\phi+ C_2 \epsilon) \Lo
  + \mathcal{O} ( c(f) )
\end{equation}
for $\sigma > c$ and $|t| \le 1$. 
We now extend the range of $s$ for which this formula is valid. 
Let $\mu$ and $\delta$ be real constants and assume that $0 \le \mu \le \frac{1-\delta}{x_0} \frac{\log \Lo}{\Lo}$
where $0<\delta < 0.5$. 
Now consider the Laplace transform  pair
\begin{align*}
  g(t) = e^{\mu t}f(t) \text{ and } 
  G(z) = \int_{0}^{\infty} e^{-zt} g(t) dt =
  F(z-\mu). 
\end{align*}
The parameter $\mu$ will allow us to move into the critical strip because of the decay properties of $g$ and its derivatives. 
Moreover, $g$ 
satisfies Condition 1 with the same $x_0$ as before.  However, 
by \eqref{eq:cond2} and \eqref{eq:cond3} 
\[
  B(g) \ll e^{\mu x_0} \Big(
   B(f) + 2 \mu (B(f)x_0+ |f(0)|x_{0}^{-1})
   + \mu^2 (B(f)x_{0}^2+ |f(0)|)  \Big).
\]
Consequently, by \eqref{eq:cf},
\[
  c(g) \ll \Lo^{1-\delta}(\log \Lo)^2 \text{ for } 0 \le \mu \le  \frac{1-\delta}{x_0} (\log \Lo)
\]   
where the implied constant depends on $x_0$ and $f$. 
By applying \eqref{eq:Kineq} to the Laplace transform pair $g$ and $G$  and
noting that  
\[
x^{-s} g(\Lo^{-1} \log x)= x^{-s+\mu/\Lo} f(\Lo^{-1}  \log x),
\]
we derive
\begin{multline*}
   \mathcal{K}(s-\mu/\Lo)
   \le \Lo \Re F((s-\mu/\Lo-1)\Lo)- \Lo\sum_{\varrho \in \mathcal{R}_{\epsilon,t}}
   \Re( F((s-\mu/\Lo-\varrho) \Lo)) 
 \\ + g(0)(\phi+ C_2 \epsilon) \Lo + O\left( c(g)  \right).
\end{multline*}
Replacing $s-\mu/\Lo$ by $s$ and noting that $g(0)=f(0)$, we obtain 
\begin{multline*}
   \mathcal{K}(s)
   \le \Lo \Re F((s-1)\Lo)- \Lo\sum_{\varrho \in \mathcal{R}_{\epsilon,t}}
   \Re( F((s-\varrho) \Lo)) 
 \\ + f(0)(\phi+ C_2 \epsilon) \Lo + O\left( \Lo^{1-\delta}(\log \Lo)^2  \right)
\end{multline*}
valid for $\sigma > c- \frac{(1-\delta)(\log \Lo)}{ x_0 \Lo}$.  Now choosing $c=1+\frac{\delta \log \Lo}{x_0 \Lo}$, 
and then replacing $2 \delta$ by $\delta$ we obtain the following. 
\begin{lemma} \label{KsE}
Let $0 < \epsilon \le 10^{-2}$ and $0 < \delta < 1$.  
Suppose $f$ satisfies Condition 1.
If  $\sigma > 1- \frac{(1-\delta)(\log \Lo)}{ x_0 \Lo}$ and $|t|\le 1$, then there exists $C_3 >0$ such that
\begin{align*}
 \mathcal{K}(s)
   & \le \Lo \Re F((s-1)\Lo)- \Lo\sum_{\varrho \in \mathcal{R}_{\epsilon,t}}
   \Re( F((s-\varrho) \Lo)) 
 + f(0)(\phi+ C_3 \epsilon) \Lo
\end{align*}
provided that $d_K$ is sufficiently large. 
\end{lemma}

\section{The  Deuring-Heilbronn Phenomenon}
Let $0< \epsilon < 0.01$, let $0 < \delta <1 $, and let $f$ be a function satisfying Conditions 1 and 2. 
Associated to $f$ is the rectangle $\mathcal{R}$  defined by
\begin{equation}
   \label{eq:R}
   \mathcal{R}= \Big\{  s \in \mathbb{C} \ \Big| \ 1- \frac{(1-\delta)(\log \Lo)}{x_0 \Lo} \le \sigma \le 1, 
        |t| \le 1 \Big\}.
\end{equation}
Suppose that $\varrho_1=\beta_1$ is a real zero in $\mathcal{R}$. Moreover, suppose that 
$\beta_1$ is the maximum real part of all zeros of $\zeta_K(s)$ in $\mathcal{R}$. We set 
\[
  \beta_1=1- \lambda_1 \Lo^{-1}. 
\]
Let $\varrho'=\beta'+i \gamma'$ denote another zero of $\zeta_K(s)$ in $\mathcal{R}$
such that $\Re (\varrho')$ is maximal.
We shall write 
\[
   \beta' = 1- \lambda' \Lo^{-1} \text{ and }
   \gamma' = \mu' \Lo^{-1}. 
\]
The goal of this section is to derive inequalities which relate $F$, $\beta_1$, and $\beta'$.
From these inequalities we derive our repulsion theorems which show that $\beta'$ is far from
1 if $\beta_1$ is very close to 1. Many proofs of repulsion theorems involving real zeros of an $L$-function
employ the inequality $1 +\cos(x) \ge 0$. We use this in the form
\[
 1+ \Re( (N \mathfrak{a})^{-i \gamma'}) \ge 0
\]
where  $\mathfrak{a}$ is a nonzero ideal of $K$.   
This implies that 
\begin{equation}
   \mathcal{K}(\beta') + \mathcal{K}(\beta'+i \gamma') \ge 0
   \label{eq:Kineq2}
\end{equation}
where $\mathcal{K}(s)$ is given by \eqref{eq:Gs}. This is
the starting point of our argument.  
In order to obtain useful information from this inequality we need to impose further conditions on $f$ and $F$. 

\noindent {\it Condition 2}. The function $f$ is non-negative and 
\begin{equation}
  \label{eq:Fpos}
  \Re(F(z)) \ge 0 \text{ for } \Re(z) \ge 0. 
\end{equation}
Before proceeding with our proof of  Theorem \ref{dh}, we will describe briefly the argument employed
in \cite{LMO}.
In fact, the authors consider the function 
\[
    \mathcal{H}_{j}(s) = \Re \Big(  \frac{d^j}{ds^j}  \Big( -\frac{\zeta_{K}'}{\zeta_K}(s) \Big) \Big)
\]
where $j \in \mathbb{N}$. Their initial observation is that a variant of Turan's second main theorem implies that 
\[
   \mathcal{H}_{2j_0-1}(2) + \mathcal{H}_{2j_0-1}(2+i \gamma') \ge 0
\]
for some $j_0 \in \mathbb{N}$. By differentiating the global explicit formula for $-(\zeta_{K}'/\zeta_K)(s)$, the left hand side 
of this inequality may be related to a sum over zeros of $\zeta_K(s)$.  
In contrast, in our argument we shall make use of the positivity condition \eqref{eq:Fpos}
to obtain a lower bound for the relevant sum over zeros.  Numerically the method of \cite{LMO} produces a smaller constant 
for the Deuring-Heilbronn phenomenon, though it is valid for a larger range of $\beta'$.  

We are now prepared to derive from \eqref{eq:Kineq2}  an inequality relating $\lambda_1$, $\lambda'$, $f$, and $F$. 
Since $\beta' > 1- \frac{1-\delta}{x_0}(\log \Lo) \Lo^{-1}$ we may apply Lemma \ref{KsE}:
\begin{align*}
  \mathcal{K}(\beta') 
   & \le \Lo F((\beta'-1)\Lo)- \Lo \sum_{\varrho \in \mathcal{R}_{\epsilon,0}}
   \Re F((\beta'-\varrho) \Lo) 
 + f(0)(\phi+ C_3 \epsilon) \Lo.
\end{align*}
By the choice of $\varrho'$, all terms $\varrho \ne \varrho_1$ produce non-negative 
contributions and may be dropped. Therefore
\[
  \mathcal{K}(\beta') \le \Lo F(-\lambda')- \Lo F(\lambda_1-\lambda')
  + f(0)(\phi+ C_3 \epsilon) \Lo.
\]
Also by Lemma \ref{KsE}
\begin{align*}
   \mathcal{K}(\beta'+i \gamma') 
    & \le \Lo \Re F((\beta'+i \gamma'-1)\Lo) \\ 
    & - \Lo\sum_{\varrho \in \mathcal{R}_{\epsilon, \gamma'}} 
   \Re F((\beta'+i \gamma'-\varrho) \Lo) 
 + f(0)(\phi+ C_3 \epsilon) \Lo.
\end{align*}
For $d_K$ large enough, we have $1-\epsilon < 1-\frac{1-\delta}{x_0} (\log \Lo) \Lo^{-1} <\beta'$ and thus $\varrho' \in \mathcal{R}_{\epsilon,\gamma'}$.
In the sum over zeros, we have a contribution from $\varrho=\varrho'$ which contributes a term $F(0)$. 
Note that for all $ \gamma'$ satisfying $|\gamma'| \le 1$, $\varrho_1$ occurs in the sum and makes a 
contribution $\Re F((\varrho'-\varrho_1)\Lo)$.  
Putting this together yields 
\[
   \sum_{\varrho \in \mathcal{R}_{\epsilon,\gamma'}}
   \Re F((\beta'+i \gamma'-\varrho) \Lo) 
   \ge F(0) + \Re F((\varrho'-\varrho_1)\Lo).
\]
We deduce that 
\begin{align*}
  \mathcal{K}(\beta'+i \gamma') 
    & \le \Lo  \left( 
    \Re F(-\lambda'+i \mu')
    -F(0) 
    - \Re F(\lambda_1-\lambda'+i\mu')  
    + f(0)(\phi+ C_3 \epsilon)
    \right).
\end{align*}
It follows from (\ref{eq:Kineq2}) that 
\begin{equation*}
  F(-\lambda') - F(\lambda_1-\lambda') + \Re  F(-\lambda'+i \mu')  - F(0)
  -  \Re  F(\lambda_1-\lambda'+i \mu')
   + f(0) (2 \phi + 2C_3 \epsilon) \ge 0.
\end{equation*}
However, we observe that 
\begin{multline*}
  \Re ( F(-\lambda'+i \mu')
  - F(\lambda_1-\lambda'+i \mu'))
  =\int_{0}^{\infty} f(t) e^{\lambda't} (1- e^{-\lambda_1 t}) 
  \cos (\mu' t) \, dt  
 \\  \le \int_{0}^{x_0} f(t) e^{\lambda' t} (1-e^{-\lambda_1 t}) dt 
   = F(-\lambda')-F(\lambda_1-\lambda').
\end{multline*}
Combining the last two inequalities, we deduce 
\begin{lemma} \label{dh1}
Let $0<\epsilon <0.01$, $0 < \delta < 1$, and $f$ satisfy Conditions 1 and 2. 
Let $\beta_1$ be an exceptional zero of $\zeta_K(s)$ in $\mathcal{R}$ defined by \eqref{eq:R} and $\varrho'=\beta'+i \gamma'$ be another zero of $\zeta_K(s)$ in $\mathcal{R}$ with $\beta'$ maximal.
Then we have 
\begin{equation}
    \label{eq:Fin}
    2 F(-\lambda')-2 F(\lambda_1-\lambda')-F(0) + f(0) (2 \phi + 2C_3\epsilon)
    \ge 0
\end{equation}
provided that $d_K$ is sufficiently large. 
\end{lemma}
With this inequality in hand we may now prove an explicit version of the Deuring-Heilbronn phenomenon. 
We shall derive lower bounds for $\lambda'$ in terms of $\lambda_1$ by choosing specific test functions
$f$ to use in \eqref{eq:Fin}.  The choice of $f$ will depend on the size of $\lambda_1$. 
There are two ranges to consider: $\lambda_1$ very small and $\lambda_1$ of medium size.
Note that when we apply Lemma \ref{dh1} below we shall replace $2C_3 \epsilon$ by $\epsilon$. 

\subsection{Case 1: $\lambda_1$ very small}   First we choose one of the simplest functions satisfying Conditions 1 and 2. 
Let 
\[
  f(t) =  \begin{cases}
     x_0-t
     & \text{ for } 0 \le t \le x_0, \\
    0 & \text{ for } t > x_0. \\
    \end{cases}
\]
For this choice we have $f(0)=x_0$ and  $F(0)= \int_{0}^{x_0} (x_0-t)dt = \frac{1}{2} x_0^2$. 
Also, since $\lambda_1$ is very small, we use the simple inequality
\[
  F(-\lambda')-F(\lambda_1 -\lambda')
  = \int_{0}^{\infty} f(t) e^{\lambda' t} (1-e^{-\lambda_1 t}) dt
  \le \lambda_1 \int_{0}^{\infty} t f(t) e^{\lambda' t} dt. 
\]
However
\begin{align*}
  \int_{0}^{x_0} t(x_0-t)e^{\lambda' t} dt 
 = \lambda'^{-3}(x_0 \lambda' e^{x_0 \lambda'}-2e^{x_0 \lambda'}
  +x_0 \lambda' +2)  
  \le \lambda'^{-3}(x_0 \lambda' e^{x_0 \lambda'})
\end{align*}
and thus
\[
   2F(-\lambda')-2F(\lambda_1-\lambda') \le 2x_0 \lambda_1 \lambda'^{-2} e^{x_0 \lambda'}.
\]
Therefore, by Lemma \ref{dh1}, we derive 
\[
    2x_0 \lambda_1 \lambda'^{-2} e^{x_0 \lambda'}-\frac{1}{2}x_0^2 + x_0
  (2 \phi + \epsilon) \ge 0. 
\]
We choose $x_0 = 4 \phi + \lambda'^{-1} +2 \epsilon$ to obtain
\[
   2 \lambda_1 \lambda'^{-2} e^{x_0 \lambda'} - \frac{1}{2} \lambda'^{-1} \ge 0.
\]
Rearranging, it follows that
\[
   \lambda_1 \ge \frac{\lambda'}{4e} \exp(-\lambda'(4 \phi+2 \epsilon))
  \ge \exp(-\lambda'(4 \phi+2 \epsilon))
\]
for $ \lambda' \ge 4e$ and $d_K$ sufficiently large.  Further, solving for $\lambda'$ leads to 
\[
    \lambda' \ge  \Big( \frac{1}{4 \phi +2 \epsilon}  \Big) \log ( \lambda_1^{-1} ).
\]
Let $0<\epsilon \le 10^{-6}$.
By the zero-free region bound \eqref{eq:kadiri} with $R=12.74$, $\lambda'^{-1} \le 12.74$ and thus 
$x_0 \le 4 \phi +R^{-1}+2\epsilon = 13.8456$.  Hence, if 
$\beta' > 1-  \frac{\log \Lo}{13.85\Lo}$, then $\beta' \in \mathcal{R}$ for this choice of $f$
and  $\delta = 10^{-6}$. 
We obtain the following lemma:
\begin{lemma} \label{dhsmall}
Let $\beta_1$ be an exceptional zero of $\zeta_K(s)$  and $\varrho'=\beta'+i \gamma'$ be another zero of $\zeta_K(s)$ 
with  $\lambda' <  \frac{1}{13.85} \log \Lo$ and $|\gamma'| \le 1$. 
Then either $\lambda' < 4e$ or
\[
    \lambda' \ge  0.9045 \log( \lambda_1^{-1})
\]
for $d_K$ sufficiently large.
\end{lemma}
Under these conditions it follows that 
\[
    \lambda' \ge  0.9045 \log( \lambda_1^{-1})>4e
\]
if $\lambda_1 < \exp\left(-\frac{4e}{0.9045}\right)
=6.015645\ldots \times 10^{-6}$.

Next we consider the case when $\lambda_1$ is slightly larger. 
\subsection{$\lambda_1$ medium size}
Notice that 
\[
  F(-\lambda')- F(\lambda_1-\lambda') = \int_{0}^{\infty} f(t) e^{\lambda' t} (1-e^{-\lambda_1 t}) dt 
\]
is increasing with respect to $\lambda' \in [\lambda_1,1/2]$. 
Therefore 
 \begin{equation}
  \label{eq:ineqsimplified}
2 F(-\lambda_1)-3 F(0) + f(0) (2 \phi + \epsilon) \ge 0
\end{equation}
implies \eqref{eq:Fin}. 
In \cite{HB}, Heath-Brown addressed the problem of minimizing the expression $F(-\lambda)$ 
with respect to $f$ for a fixed parameter $\lambda$ and  fixed values of $F(0)$ and $f(0)$.  
Moreover, he determined non-rigorously the optimal such $f=f_{\lambda, \theta}$, which depends on $\lambda$ and another parameter $\theta$. 
With this function in hand, the minimization of \eqref{eq:ineqsimplified} is equivalent to minimizing the expression 
\[
   \frac{F(-\lambda)-\frac{3}{2}F(0)}{f(0)}
\]
with respect to $\theta$.  It turns out that the optimal $\theta$  is the unique solution to 
\[
   \sin^2(\theta) = \tfrac{3}{2}(1-\theta \cot(\theta)),
\]
in $(0, \frac{\pi}{2})$ and  is given by 
\begin{equation}
  \label{eq:thetavalue}
\theta = 1.272979\ldots .
\end{equation}
The function $f$ is defined as follows (see  \cite[Lemma 7.1]{HB}):
$f$ is chosen of the form $f  = g*g$ to ensure  the positivity condition \eqref{eq:Fpos}
for the Laplace transform $F$.
In order to abbreviate notation, put 
\[
\zeta= \lambda \tan(\theta) \text{ and } d_{\theta,\lambda}= \theta \zeta^{-1}= \frac{\theta}{\lambda \tan(\theta)}.
\]
We define 
 \[
     g(t) = g_{\lambda, \theta} (t)= 
     \begin{cases}
     \lambda(1+\tan^2(\theta))(\cos(\zeta t)-\cos\theta)
     & \text{ for } |t| \le d_{\theta,\lambda} , \\
    0 & \text{ for } |t| \ge d_{\theta,\lambda}. \\
    \end{cases}
 \]
More explicitly, we have 
\begin{equation}
\begin{split}
   \label{eq:ftdefn}
    f(t) =  f_{\lambda, \theta}(t) = \lambda(1+\tan^2(\theta)) \Big(
    & \lambda (1+\tan^2(\theta))(d_{\theta,\lambda}-\frac{t}{2}) \cos(\zeta t)
    + \lambda (2 d_{\theta,\lambda}-t) \\
    & + \frac{\sin(2 \theta - \zeta t)}{\sin(2 \theta)}
    - 2 \Big(  
    1 + \frac{\sin(\theta-\zeta t)}{\sin \theta}
    \Big) \Big)
\end{split}
\end{equation}
for $0 \le t \le 2 d_{\theta,\lambda}$ and $f(t) =0$ for $t \ge 2 d_{\theta,\lambda}$.  We have the 
specific values 
\begin{align*}
   f(0) & = \lambda(1+\tan^2(\theta))(\theta \tan(\theta)+3 \theta \cot(\theta)-3), \\
   F(0) & = 2(1+\tan^2(\theta))(1-\theta \cot(\theta))^2. 
\end{align*}
We shall apply $f=f_{\lambda,\theta}$ in \eqref{eq:Fin} for specific values of $\lambda$. 
It might be reasonable to expect that such a function would be close to optimal in the inequality \eqref{eq:Fin}.

We now analyze \eqref{eq:Fin} for various choices of $\lambda$. 
Let $b$ be a fixed positive number and select a fixed positive parameter $\lambda= \lambda_b$.  
This choice of $\lambda_b$  completely determines the function $f= f_{\lambda_b,\theta}$ given by \eqref{eq:ftdefn}. 
For simplicity, we denote
\[
h(\lambda_1,\lambda') = 2 F(-\lambda')-2 F(\lambda_1-\lambda')-F(0) + 2 \phi  f(0) .
\]
Note that $F(-\lambda')-F(\lambda_1-\lambda')$ increases with respect to $\lambda_1$
and thus $h(\lambda_1,\lambda')$ also
increases with respect to $\lambda_1$. 
Hence, for all $\lambda_1$ satisfying $0\le \lambda_1 \le b$, we have
\begin{equation}
\label{eq-trig}
0 \le 
h(\lambda_1,\lambda')
\le h(b ,\lambda').
\end{equation}
Let $\lambda_b'$ be the solution of  
\begin{equation}
  \label{eq:hblamp}
h(b ,\lambda')= 0.
\end{equation}
In fact, it may be shown that for all $\delta' >0$ sufficiently small,
$\lambda' \ge \lambda_b'-\delta'.$  If not, there exists $\delta'' >0$ such that $\lambda' < \lambda_b'-\delta''$.
Therefore $h(b,\lambda') < h(b,\lambda_b'-\delta'')<0$. Next choose $\epsilon$ so small that 
$\epsilon f(0) < |h(b,\lambda'-\delta'')|$.  This implies that 
\[
   h(b,\lambda')+\epsilon f(0) < h(b,\lambda_b'-\delta'')+\epsilon f(0) < 0
\]
which contradicts \eqref{eq:Fin} for $\epsilon$ sufficiently small. 
Thus if $d_K$ is sufficiently large with respect to $\epsilon$ and $b$ and $0 \le \lambda_1 \le b$,
then $\lambda' \ge \lambda_{b}' -\delta'$ for all $\delta'$ sufficiently small.
We choose $\epsilon$ such that $\delta' \le 10^{-6}$.
We have the following table (the values are truncated at the first 4 digits).
\begin{center}
{\bf Table 1}
\end{center}
\begin{center}
\begin{tabular}{|c|c|c|c|c|}
\hline
 $b$ &   $\lambda=\lambda_b$ & $\lambda_{b}'$  & $\log(b^{-1})$  & $x_0=x_0(b)$ \\
\hline
      $10^{-6}$ &   0.543 &  12.3982 & 13.8155 & 1.4391\\
\hline
      $10^{-5}$ &  $ 0.537$ & $10.3716$ & $11.5129$ & 1.4552\\
\hline
   $10^{-4}$ &    0.526 &   8.2848 &  9.2103 & 1.4856 \\
 \hline
   $10^{-3}$ &    0.509 &  6.1120 &   6.9077 & 1.5353\\
\hline
   $0.005$  & 0.490  & 4.5233 & 5.2983 & 1.5948 \\
\hline
   $0.01$ & $0.477$ & $3.8182$ & $4.6051$ & 1.6383 \\
\hline
    $0.02$ & $0.462$ & $3.1007$ & $3.9120$ & 1.6914 \\
\hline
    $0.03$ & $0.450$ & $2.6764$ & $3.5065$ & 1.7366 \\
\hline
    $0.04$ & $0.441$ & $2.3740$ & $3.2188$ & 1.7720 \\
\hline
    $0.05$ & $0.433$ & $2.1391$ & $2.9957$ & 1.8047 \\
\hline
   0.055 &    0.429 &    2.0389 & 2.9004   & 1.8216 \\
\hline
    $0.06$ & $0.426$ & $1.9474$ & $2.8134$ &1.8344 \\
\hline
  0.065 &      0.422 &   1.8634 &     2.7333 & 1.8518\\
\hline
    $0.07$ & $0.419$ & $1.7857$ & $2.6592$ & 1.8650\\
\hline 
  0.071  & $0.418$ & $1.7708$ & $ 2.6450$ & 1.8695  \\
\hline 
  0.072  & $0.418$ & $1.7562$ & $ 2.6310$ & 1.8695\\
\hline
  0.073  & $0.417$ & $1.7418$ & $2.6172$ & 1.8740 \\
\hline
 0.074 &  $0.416$ & $1.7275$ & $2.6036$ & 1.8785 \\
\hline
 0.075 &  $0.416$ & $1.7135$ & $ 2.5902$ & 1.8785  \\
\hline
 0.076 & $0.415$ & $ 1.6996$ & $ 2.5770$ & 1.8830 \\
 \hline 
 0.077 & $0.415$ & $1.6860$ & $ 2.5639$ & 1.8830 \\
 \hline 
  0.078 & $0.414$ & $1.6725$ & $ 2.5510$ & 1.8876 \\
     \hline 
    $\frac1{12.74}=0.0784\ldots$ & $0.413$ & $1.6659$ & $ 2.5447 $ &  1.8921 \\
\hline
\end{tabular}
\end{center}
Recall that $\lambda=\lambda_b$ determines the function $f=f_{\lambda_b,\theta}$ defined by \eqref{eq:ftdefn}, $\lambda_b'$ is given by \eqref{eq:hblamp},
and $x_0=x_0(b)= \frac{2\theta}{\lambda_b \tan(\theta)}$ determines the support of $f=f_{\lambda_b,\theta}$. 
For each value of $b$, we determined $\lambda_b$ by numerical experimentation. 
From this table we are now able to derive a lower bound for $\lambda'$.   For instance if $\lambda_1 \in [0.078 , \frac{1}{12.74}]$,
then 
\[
  \lambda'  \ge  1.6659 \ge 1.6659 \frac{\log(\lambda_{1}^{-1})}{\log(0.078^{-1})}  \ge 0.6546   \log (\lambda_{1}^{-1}).
\]
By applying the same argument to each of the subintervals  $[10^{-6},10^{-5}]$, $[10^{-5},10^{-4}]$, $\ldots$, $[0.076,0.077]$, and $[0.077,0.078]$,  we 
deduce that $\lambda' \ge 0.6546   \log (\lambda_{1}^{-1})$ in each of these subintervals. 
Moreoever, glancing at the final column of the table we see that in each case $x_0=x_0(b)$ is bounded by $1.8922$. 
Consequently, if $\beta' > 1-\frac{1}{1.9} (\log \Lo) \Lo^{-1}$ (or $\lambda' <  \frac{1}{1.9} \log \Lo$), then $\beta' \in \mathcal{R}$ for each choice of $f= f_{\lambda_b,\theta}$
and $\delta \le 10^{-6}$.  We thus derive our final lemma. 

\begin{lemma} \label{dhmed}
Let $\beta_1$ be an exceptional zero of $\zeta_K(s)$  and $\varrho'=\beta'+i \gamma'$ be another zero of $\zeta_K(s)$ 
with  $\lambda' <   \frac{1}{1.9} \log \Lo$ and $|\gamma'| \le 1$.  If $10^{-6} \le \lambda_1 \le \frac1{12.74}$, then we have
\[
     \lambda' \ge 0.6546 \log (\lambda_{1}^{-1})  
\]
for $d_K$ sufficiently large.
\end{lemma}
Combining Lemmas \ref{dhsmall} and \ref{dhmed} completes the proof of Theorem \ref{dh}. 

\noindent {\bf Acknowledgements} 

\noindent The first author is supported in part by a University of Lethbridge ULRF grant and the second author is supported in part by an NSERC Discovery grant.

\end{document}